\documentclass[12pt,a4paper]{amsart}

\usepackage[utf8]{inputenc}
\usepackage{lmodern}
\usepackage[T1]{fontenc}
\usepackage{graphicx}
\usepackage{amsthm}
\usepackage{amsmath}
\usepackage{amsfonts}
\usepackage{amssymb}
\usepackage[margin=2.50cm]{geometry}
\usepackage[shortlabels]{enumitem}
\usepackage[dvipsnames]{xcolor}
\usepackage{algpseudocode}
\usepackage[draft]{todonotes}
\usepackage{mathtools}
\usepackage{verbatim}
\usepackage{hyperref}
\usepackage{tikz-cd}
\usepackage{vwcol}

\definecolor{rdeca}{RGB}{144,26,30}
\definecolor{modra}{RGB}{0, 0, 120}
\definecolor{temna}{RGB}{0, 0, 90}
\hypersetup{colorlinks,
	linkcolor={rdeca},
	citecolor={modra},
	urlcolor={temna},
}



\def\acts{\curvearrowright}

\newcommand{\res}{\upharpoonright}
\newcommand{\setdef}{\;\vert\;}
\newcommand{\sq}{\subseteq}

\newcommand{\F}{\mathbb{F}}

\newcommand{\Z}{\mathbb{Z}}

\newcommand{\G}{\mathcal{G}}

\newcommand{\trik}{\triangleleft}

\newcommand{\baire}{\omega^\omega}

\newcommand{\injfin}{(\omega)^{<\omega}}
\newcommand{\injinf}{(\omega)^\omega}
\newcommand{\injfininf}{(\omega)^{\leq \omega}}

\DeclareMathOperator{\dom}{dom}

\DeclareMathOperator{\lh}{lh}
\DeclareMathOperator{\id}{id}

\DeclareMathOperator{\fix}{fix}
\DeclareMathOperator{\range}{range}
\DeclareMathOperator{\is}{Is}
\def\acts{\curvearrowright}

\theoremstyle{plain}
\newtheorem{theorem}{Theorem}[section]
\newtheorem{lemma}[theorem]{Lemma}
\newtheorem{proposition}[theorem]{Proposition}

\newtheorem{question}[theorem]{Question}

\newtheorem{claim}[theorem]{Claim}
\newtheorem{subclaim}[theorem]{Subclaim}
\newtheorem{fact}[theorem]{Fact}

\theoremstyle{definition}
\newtheorem{definition}[theorem]{Definition}

\newcommand{\Asterisk}{\mathop{\scalebox{1.6}{\raisebox{-0.15ex}{$\ast$}}}}%

\title[Definability of mcgs]{Definability of maximal cofinitary groups}
\author[Mejak]{Severin Mejak}
\address{Department of Mathematical Sciences\\ University of Copenhagen\\ Universitetsparken 5\\ 2100 Copenhagen, Denmark}
\email{seme@math.ku.dk}

\author[Schrittesser]{David Schrittesser}
\address{Institute for Advanced Study in Mathematics\\ Harbin Institute of Technology\\
92 West Da Zhi Street\\ Harbin, Heilongjiang 150001\\ China \emph{and}\vspace{-0.3cm}}
\address{University of Toronto\\ 40 St.\ George Street\\ Toronto\\ Ontario\\ Canada M5S 2E4}
\email{david.schrittesser@univie.ac.at}



\begin{document}
    \maketitle
    
    \begin{abstract}
        We present a proof of a result, previously announced by the second author,
        that there is a closed (even $\Pi^0_1$) set generating an $F_\sigma$ (even $\Sigma^0_2$)  maximal cofinitary group (short, mcg) which is isomorphic to a free group.
        In this isomorphism class, this is the lowest possible definitional complexity of an mcg.
    \end{abstract}
    
    \section*{Introduction}
    
    In \cite{cameron}, Cameron introduced the notion of a cofinitary subgroup of $S_\infty$ as follows. An element $g \in S_\infty \setminus \{\id_\omega\}$ is called \emph{cofinitary} if it has only finitely many fixed points, i.e., there is some $n \in \omega$ so that for every $m > n$ it holds that $g(m) \neq m$. Then a subgroup $G \leq S_\infty$ is \emph{cofinitary} if every $g \in G \setminus \{\id_\omega\}$ is cofinitary. We write ``cofinitary group'' in place of the longer ``cofinitary subgroup of $S_\infty$''. Cofinitary groups were studied before \cite{cameron} under different names; e.g., in \cite{adeleke} they are called \emph{sharp}. See \cite{cameron} for combinatorial properties of cofinitary groups and \cite{truss} and \cite{adeleke} for results on their embeddings.
    
    A cofinitary group is \emph{maximal} (we write \emph{mcg} for short) if it is not strictly contained in any cofinitary group. Mcgs were first considered in \cite{truss} and \cite{adeleke}. The first result which made mcgs interesting to set-theorists was established by Adeleke in \cite{adeleke} and asserted that mcgs are always uncountable. Current research on mcgs is divided between answering two questions about mcgs. The first question asks what are the possible cardinalities of mcgs, and in particular, what is the least possible size of an mcg. See e.g. \cite{brendle_spinas_zhang} and \cite{hrusak_steprans_zhang} for early results which answer parts of this question. The second question asks what is the least achievable complexity of an mcg. This paper concerns itself with the latter.
    
    The first breakthrough in the study of definability of mcgs was when Gao and Zhang (see \cite{gao_zhang}) established that if $V = L$ there is a $\Pi^1_1$ set which generates an mcg. This was improved by Kastermans in \cite{kastermans_compact} to the existence of a $\Pi^1_1$ mcg under the assumption $V = L$. Kastermans also proved the beautiful results that no mcg can be $K_\sigma$ (see \cite{kastermans_compact}) and that no mcg can have infinitely many orbits (see \cite{kastermans_orbits}).

    Kastermans' result on definability inspired \cite{fischer_schrittesser_tornquist_mcg}, in which Fischer, T\"{o}rnquist and the second author constructed a Cohen-indestructible $\Pi^1_1$ mcg. The next milestone was achieved (to some surprise) soon afterwards, when Horowitz and Shelah established in \cite{horowitz_shelah} that it is provable in $\mathsf{ZF}$ (without the axiom of choice) that there is a Borel mcg. Using ideas of \cite{horowitz_shelah}, the second author proved (again without use of the axiom of choice) in \cite{constructing_mcg} that there is actually an arithmetical mcg, and announced that the construction from \cite{constructing_mcg} can be improved to produce an $F_\sigma$ (actually even $\Sigma^0_2$) mcg.
    
    In this paper we use the idea of the construction from \cite{constructing_mcg} and enhance it with ideas from \cite{compactness_med} to get a construction (not employing the axiom of choice) proving the following.
    
    \begin{theorem}\label{sigma_mcg}
        There is a $\Pi^0_1$ subset of $\baire$ which freely generates a $\Sigma^0_2$ mcg.
    \end{theorem}

    Since the topological interior of any cofinitary group is clearly empty, $\Pi^0_1$ is the best possible complexity of a set generating an mcg.

    By a result of Dudley (\cite{dudley}), there is no Polish topology on a group freely generated by continuum many generators. This clearly implies that $\Sigma^0_2$ is the best possible complexity of a freely generated mcg. Additional difficulties on potential constructions of $G_\delta$ mcgs are imposed by \cite{slutsky_automatic_cont}, in which Slutsky improved on Dudley's result and proved that if $G$ is a free product of groups and carries a Polish topology, then $G$ is countable. Since all presently known constructions of definable mcgs produce groups which decompose into free products, this means that current ideas are insufficient to produce a $G_\delta$ mcg.

    Finally, we introduce \emph{finitely periodic groups} (a relative to cofinitary groups) and prove an analogue to a well known theorem for mcgs.
    
    \subsection*{Structure of the paper}
    
    We begin by briefly recalling the notation we will be using. In Section \ref{constructing} we prove Theorem \ref{sigma_mcg}.
    We proceed with Section \ref{limitations}, where we discuss known obstructions to constructing $G_\delta$ mcgs.
    We conclude the paper with Section \ref{open_problems}, where we introduce finitely periodic groups and discuss open problems related to mcgs and maximal finitely periodic groups.
    
    \subsection*{Acknowledgments}

    The authors are thankful to Asger T\"{o}rnquist for referring them to \cite{dudley}, which led the authors to other work that built upon it. The results from these references form the core of Section \ref{limitations}.
    
    While working on this paper, S.~M.~was on a longer research visit to Caltech and is greatly appreciative to Alekos Kechris and other members of the Caltech logic group for valuable discussions and for making the visit possible. He was also on a shorter visit to the University of Vienna and is thankful to Vera Fischer for hosting him.
    
    The first author wishes to thank Asger T\"{o}rnquist for valuable conversations on related topics. He also thanks Andrea Bianchi for suggesting that we consider (what we now call) finitely periodic groups.
    
    S.~M.~gratefully acknowledges support from the following grant.
    
    \begin{center}
    \begin{vwcol}[widths={0.18, 0.82}, justify= flush, rule = 0pt]
    \hspace{0.5cm}
	\includegraphics{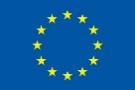}
	\vfill\eject
	\phantom{ }\hspace{-0.6cm}
	\vspace{-0.3cm}
	\phantom{ }\\
	\scriptsize{\noindent This project has received funding from the European Union’s Horizon 2020 research and\\innovation programme under the Marie Sk\l{}odowska-Curie grant agreement No 801199}
    \end{vwcol}
    \end{center}

    D.~S.\ acknowledges the support of the Government of Canada’s New Frontiers in Research Fund (NFRF), through project grant number 
    NFRFE-2018-02164.

    \section{Notation}
    
    We use $\omega$ to denote the set of all natural numbers. We think of natural numbers as ordinals, so for $n \in \omega$, $n = \{0, 1, \ldots, n-1\}$. No other properties of ordinals will be used and no knowledge of ordinals is required. For a set $A$, we let $|A|$ denote its cardinality.
    
    For a finite sequence $s$, we denote its length by $\lh(s)$. We also use $s(-1)$ to denote the last entry of $s$, i.e., $s(-1) := s(\lh(s) - 1)$. For a set $A$, we denote the set of all finite sequences in $A$ by $A^{< \omega}$, the set of all finite or infinite sequences in $A$ by $A^{\leq \omega}$ and the set of all infinite sequences in $A$ by $A^\omega$. We denote the subsets of these sets consisting only of \emph{injective sequences}, viz., sequences where no two entries are the same, by $(A)^{< \omega}$, $(A)^{\leq \omega}$ and $(A)^\omega$ respectively. 
    For $\alpha$ in $\omega + 1 = \omega \cup \{\omega\}$ we let $[A]^\alpha$ denote the collection of all subsets of $A$ of size $\alpha$, $[A]^{<\alpha}$ denote the collection of all subsets of $A$ of size less than $\alpha$ and $[A]^{\leq\alpha}$ denote the collection of all subsets of $A$ of size at most $\alpha$. For sequences $s, t \in A^{< \omega}$ (or in $(A)^{< \omega}$) we denote $t$ strictly end-extending $s$ by $s \sqsubset t$, and non-strictly by $s \sqsubseteq t$. Generally, we use $\bar{(\cdot)}$ to denote finitary versions of infinite objects, e.g., if $f \in \injinf$, then $\bar{f} \in \injfin$ represents some initial segment of $f$.
    
    If $s \in 2^{\leq \omega}$, we use $\hat{s}$ to denote the strictly increasing function enumerating $s^{-1}[\{1\}]$, the set which $s$ maps to 1.
    Sometimes, we write finite sequences $s \in 2^{< \omega}$ in ``binary format'', e.g., $(0, 1, 0, 0, 1)$ is written as $01001_2$. This is convenient when there are many consecutive zeros in a sequence, e.g., for $n, m \in \omega$,
    \[
    (\underbrace{0, \ldots, 0}_n, 1, \underbrace{0, \ldots, 0}_m, 1)
    \]
    is then written simply as $0^n 1 0^m 1_2$.
    
    As usually, the expression $(\forall^\infty n \in \omega)\, \phi(n)$ means that $\phi(n)$ holds for all but finitely many $n$, and the expression $(\exists^\infty n \in \omega)\, \phi(n)$ means that $\phi(n)$ holds for infinitely many $n$.

    For a group $G$ and $A \sq G$, we denote the subgroup of $G$, generated by $A$, with $\langle A \rangle$. To avoid confusion, we use the round brackets for tuples (and sequences), e.g., instead of the standard $\langle a, b \rangle$, we write $(a, b)$.
    
    We use $S_\infty$ to denote the Polish group of all permutations of $\omega$ and let $\id_\omega$ denote the unit of this group. For $g_0, \ldots, g_l \in S_\infty$ and $m \in \omega$, we define the \emph{path of $m$ under $g_l \circ \cdots \circ g_0$} to be the sequence $(m(i))_{i \in l + 1}$, inductively defined by setting $m(0) := m$ and for $i \leq l$ by $m(i + 1) := g_i(m(i))$. We will sometimes omit $\circ$ and write just $g_l \cdots g_0$.
    
    We denote the free group generated by a set $A$ with $\F(A)$. For $x_0, \ldots, x_k \in A$ and $i_0, \ldots, i_k \in \{-1, 1\}$, we occasionally use the vector notation $\vec{x}$ for the reduced word $x_k^{i_k} \cdots x_0^{i_0} \in \F(A)$ in order to emphasise that $\vec{x}$ is not necessarily a generator but a proper word (in case $k > 0$).
    
    Otherwise, we use the standard notation of \cite{cdst} and introduce any non-standard notation on the way.

    We use $\dashv$, instead of the otherwise used $\square$, to denote that the proof of a nested claim or a subclaim has been completed. 
    
    \section{Constructing a \texorpdfstring{$\Sigma^0_2$}{Sigma-0-2} mcg}\label{constructing}
    
    In this section we present a construction of a $\Sigma^0_2$ mcg. Our construction and presentation is based on \cite{constructing_mcg}, from which we also borrow some notation. Note that many of these notions have a slightly different definition in order to accommodate new coding components and in this way make the construction more definable. The construction of the definable mcg and the proofs of maximality and cofinitariness closely follow the respective proofs in \cite{constructing_mcg}. The main improvement over the construction from \cite{constructing_mcg} is that we augment it by ideas from \cite{compactness_med}. This will be apparent in the definition of our group and when proving that the constructed group is definable.
    
    It is not necessary for the reader to be familiar with either \cite{constructing_mcg} or \cite{compactness_med}. However, \cite{constructing_mcg} does occasionally contain more details, which the reader might find helpful in understanding this proof and \cite{compactness_med} contains the original idea of how maximal objects can be very definable in a simpler setting of maximal eventually different families, possibly making it easier for the reader to understand the present construction.
    
    \subsection{Cofinitary action, freely generated by continuum many generators.}
    
    We start by defining an injective map
    \[
    e: 2^\omega \times 2^\omega \times 2^\omega \to S_\infty,
    \]
    such that $\range(e) \sq S_\infty$ freely generates a cofinitary group. We follow the idea of the construction from \cite{constructing_mcg}, additionally making sure that the intervals are wider, allowing us to code the second and third component of the domain as well.

    Recall that we denote the free group, generated by $2^\alpha \times 2^\alpha \times 2^\alpha$, by $\F(2^\alpha \times 2^\alpha \times 2^\alpha)$ for any $\alpha \in \omega + 1\, (= \omega \cup \{\omega\})$. For a word
    \[
    w = (x_k, d^0_k, d^1_k)^{i_k} \cdots (x_0, d^0_0, d^1_0)^{i_0} \in \F(2^\alpha \times 2^\alpha \times 2^\alpha)
    \]
    we sometimes write $w = (w_0, w_1, w_2)$, where we implicitly define
    \[
    w_0 := x_k^{i_k} \cdots x_0^{i_0} \in \F(2^\alpha) \quad \text{and} \quad w_{j + 1} := (d^j_k)^{i_k} \cdots (d^j_0)^{i_0} \in \F(2^\alpha)
    \]
    for $j \in 2$. For a word $w$, we denote its length by $\lh(w)$. If $m \leq \alpha$ we let $r^\alpha_m$ be the group homomorphism
    \[
    r^\alpha_m : \F(2^\alpha \times 2^\alpha \times 2^\alpha) \to \F(2^m \times 2^m \times 2^m),
    \]
    which is defined on a generator $(x, d^0, d^1) \in \F(2^\alpha \times 2^\alpha \times 2^\alpha)$ as
    \[
    r^\alpha_m(x, d^0, d^1) = (x\res m, d^0 \res m, d^1 \res m).
    \]
    We will define a sequence of finite groups
    \[
    (G_n)_{n \in \omega}
    \]
    and group homomorphisms
    \[
    e_n: \F(2^n \times 2^n \times 2^n) \to G_n,
    \]
    along with transitive faithful actions
    \[
    \sigma_n : G_n \acts I_n,
    \]
    where each $I_n = [m_n, m_{n+1})$ is a non-empty interval in $\omega$, with $m_0 = 0$ and so that the intervals constitute a partition of $\omega$.
    We also define for each $n \in \omega$ the set
    \[
    W_n := \{ w \in \F(2^n \times 2^n \times 2^n) \, | \, \lh(w) \leq n \}.
    \]
    The requirements that we wish to satisfy are the following.  We will inductively construct $(G_n)_{n \in \omega}$, $(e_n)_{n \in \omega}$ and $(\sigma_n: G_n \acts I_n)_{n \in \omega}$ so that in addition to the above, for every $n \in \omega$ it holds that
    \begin{enumerate}
        \item $\sum_{m < n} |I_m| < |I_n|$;
        \item $e_n \res W_n$ is injective;
        \item $|I_0| \geq 7$.
    \end{enumerate}
    Since $2^0 = \{\emptyset\}$ and $W_0 = \{\emptyset\}$, it is easy to define $G_0$, $e_0$, $I_0$ and $\sigma_0: G_0 \acts I_0$ so that (3) is satisfied.
    
    Suppose that we have already defined $G_n, e_n, I_n$ and $\sigma_n: G_n \acts I_n$. Let $(w^i)_{i \in l}$ be some enumeration of $W_{n+1}$ for which $w^0 = \emptyset$, the empty word. We let $e_{n+1}^0$ be a function from $2^{n+1} \times 2^{n+1} \times 2^{n+1}$ to the set of partial injections from $l = \{0, \ldots, l-1\}$ to itself, as follows. For $(x, d^0, d^1) \in 2^{n+1} \times 2^{n+1} \times 2^{n+1}$ and $i, j \in l$ we set
    \[
    e_{n+1}^0(x, d^0, d^1)(i) = j \quad \text{if and only if} \quad w^j = (x, d^0, d^1) w^i.
    \]
    Let $e_{n+1}^1(x, d^0, d^1)$ be some extension of $e_{n+1}^0(x, d^0, d^1)$ to a permutation of $l$ and let
    \[
    G_{n+1}^0 := \langle e_{n+1}^1(x, d^0, d^1) \, | \, (x, d^0, d^1) \in 2^{n+1} \times 2^{n+1} \times 2^{n+1} \rangle \leq S_l,
    \]
    where $S_l$ denotes the permutation group of $l$.
    By the universal property of the free group, $e_{n+1}^1$ extends uniquely to a group homomorphism from $\F(2^{n+1} \times 2^{n+1} \times 2^{n+1})$ onto $G_{n+1}^0$, which we denote in the same way. Since $e_{n+1}^1(w^i)(0) = i$ for every $i \in l$, $e_{n+1}^1$ is injective on $W_{n+1}$.

    For some sufficiently large $k \in \omega$ let
    \[
    G_{n+1} := G_{n+1}^0 \times S_k.
    \]
    By ``sufficiently large'' we mean large enough so that when $I_{n+1}$ is defined to be of the same size as $G_{n+1}$, condition (1) holds.
    Moreover, let
    \[
    e_{n+1}%
    : \F(2^{n+1} \times 2^{n+1} \times 2^{n+1}) \to G_{n+1}
    \]
    be defined on a generator $(x, d^0, d^1) \in 2^{n+1} \times 2^{n+1} \times 2^{n+1}$ by
    \[
    e_{n+1}(x, d^0, d^1) := (e_{n+1}^1(x, d^0, d^1), 1_{S_k} ). 
    \]
    Let $I_{n+1}$ be the interval immediately to the right of the previously defined $I_n$, which has the same size as $G_{n+1}$. Fix some bijection $\Phi_{n+1}: G_{n+1} \to I_{n+1}$ and define the action
    \[
    \sigma_{n+1}: G_{n+1} \acts I_{n+1}
    \]
    by
    \[
    \sigma_{n+1}(g)(k) = \Phi_{n+1}(g \cdot (\Phi_{n+1})^{-1}(k)).
    \]
    Clearly, all desired conditions have been met and with this the inductive construction is completed.
    
    With the sequences $(G_n)_{n \in \omega}$, $(e_n)_{n \in \omega}$ and $(\sigma_n: G_n \acts I_n)_{n \in \omega}$ at our disposal define
    \[e: \F(2^\omega \times 2^\omega \times 2^\omega) \to S_\infty\]
    by defining it on a generator $(x, d^0, d^1) \in 2^\omega \times 2^\omega \times 2^\omega$ as
    \[e(x, d^0, d^1)\res I_n := \sigma_n(e_n(r^\omega_n(x, d^0, d^1))).\]
    
    The following proposition is an elaboration of Propositions 1.3 and 1.4 from \cite{constructing_mcg}.
    
    \begin{proposition}
        The map $e$ is a continuous injective homomorphism, whose range is a cofinitary group. Moreover, $e$ is $\Delta^0_1$ (in terms of that for $w \in \F(2^\omega \times 2^\omega \times 2^\omega)$ and $n \in \omega$, we can calculate $e(w)(n)$ in finite time by analysing $r_m^\omega(w)$ for some $m \in \omega$, which can in turn be calculated in finite time from $n$; all this using $w$ as an oracle), its range is $\Pi^0_1$ and the sequences $( G_n)_{n \in \omega}$ and $(\sigma_n)_{n \in \omega}$ are $\Delta^0_1$.
    \end{proposition}
    
    \begin{proof}
        Clear by construction.
    \end{proof}

    \subsection{The definition of \texorpdfstring{$B_0$}{B0}}

    In this subsection we define the assignment $B_0$, which will form the core of our construction.

    We first introduce the map $\chi: \injfininf \to 2^{\leq\omega}$ defined on $h \in \injfininf$ by
    \[
    \chi(h) := (\underbrace{0, \ldots, 0}_{h(0)}, 1, \underbrace{0, \ldots, 0}_{h(1)}, 1, \underbrace{0, \ldots, 0}_{h(2)}, 1, \ldots) \in 2^{\leq\omega}.
    \]
    Of course, $\lh(\chi(h))$ is infinite if and only if $\lh(h)$ is infinite.

    We define the map $\chi^{\dagger} : 2^{\leq\omega} \to \injfininf$ on $x \in 2^{\leq\omega}$ by
    \[
    \chi^{\dagger}(x) := \begin{cases}
        h & \text{if } x \in \range(\chi) \text{ and } \chi(h) = x;\\
        s & \parbox[t]{8cm}{if $x \notin \range(\chi)$ and $n$ is maximal such that $x \res n \in \range(\chi)$ and $\chi(s) = x \res n$.}
    \end{cases}
    \]
    Of course, $\chi^{\dagger}$ is a left inverse of $\chi$. There are two ways for $x \in 2^\omega$ not to be in $\range(\chi)$. The first one is that $x$ is of the form
    \[
    (\ldots, 1, \underbrace{0, 0, \ldots, 0}_n, 1, \ldots, 1, \underbrace{0, 0, \ldots, 0}_n, 1, \ldots),
    \]
    in which case it clearly cannot be the image of any $g \in \injfininf$, since it violates injectivity. The second option is that $x$ only has finitely many 1s (but does not fall under the first option). Then if $n>0$, $n-1$ is the last index of a $1$ in $x$; and $n=0$ if and only if $x$ is constant with value $0$, in which case $\chi^{\dagger}(x)$ is the empty function.

    Next, define for every $g \in (\omega)^{\leq \omega}$ a function $\vartheta_g \in \injfininf$ 
    with the following properties:
    \begin{enumerate}[(i)]
        \item if $g \in \injinf$, then as $k$ increases, the maps $\vartheta_{g \upharpoonright k} \in \injfin$ approach $\vartheta_g \in \injinf$;
        \item if $g \neq g'$, then $\range(\vartheta_g)$ and $\range(\vartheta_{g'})$ are almost disjoint (this will be used in Claim \ref{ad});
        \item for any $n \in \range(\vartheta_g)$, if $n \in I_m$ and $g(n) \in I_{m'}$ then $m \leq m'$ (this will be important in Subclaim \ref{subclaim_lowest_interval} and in Proposition \ref{sigma_range}).
        \item each set $\range(\vartheta_g)$ is $g$-\emph{spaced}, i.e., for any distinct $n, n' \in \range(\vartheta_g)$, it holds that $n'$ is not in the same interval as any of $n$, $g^{-1}(n)$ or $g(n)$. This will be used in Fact~\ref{f.spaced.used1}.
    \end{enumerate}
    
    These properties are not difficult to arrange, but since our aim is to do this in such a way that in the end the definition of the set of generators is simple, we shall go into the details.

    Let
    \[
    \#: \injfin \to \omega
    \]
    be some fixed $\Delta^0_1$ bijection, which we use to introduce an auxiliary function $F$, defined for $g \in \injfin$ and $k \in \omega$ by
    \[
    F(g, k) := 2^{\#(g)} \cdot 3^{k}.
    \]
    The only property we require of $F$ is that it is $\Delta^0_1$ and injective.
    We also define an auxiliary family of functions which satisfy all of the properties (i) to (iv) above, except possibly (iii): 
    given $g \in (\omega)^{\leq \omega}$, define $\xi_g \in \injfininf$ inductively as follows. Suppose that we have defined $\xi_g$ on all $k < n$ and that for every $k < n$ we have 
     \[
    I_{F(g \res (k + 1), \xi_g(k))} \subseteq \dom(g) \cup \range(g).
    \]
    (If the above requirement is not fulfilled for some $k < n$, then $\dom(\xi_g) = n$, and our inductive definition terminates.) Then we set $\xi_g(n)$ to be the smallest such that 
    \[
    \min I_{F(g \res (n+1), \xi_g(n))}
    \]
    is strictly larger than any member of
    \[
    \bigcup \big\{\{l, g(l)
    , g^{-1}(l)
    \} \, \big| \, (\exists k < n)\,  l \in I_{F(g \res (k+1), \xi_g(k))}\big\}.
    \]
    With $\xi_g$ at our disposal,
    we finally define the function $\vartheta_g \in \injfininf$, defined on those $n \in \omega$, for which 
    \[
    \dom(g) \geq \max\big\{n+1, \min I_{F(g \res (n+1), \xi_g(n)) + 1}\big\}
    \]
    by
\[
 \vartheta_g(n) := \min \big(I_m \setminus \cup\{ g^{-1}[I_k] \, | \, k < m\}\big),
\]
where $m = F(g \res (n + 1), \xi_g(n))$. It is clear that properties (i) to (iv) are fulfilled.

Now we define for every $g \in (\omega)^{\leq \omega}$ and $c^0, c^1 \in 2^{\leq \omega}$ with $\lh(g) \leq \lh(c^0) = \lh(c^1)$ the sets
    \[
    D(g) := \dom(g) \cap \range(\vartheta_g)
    \]
    and
    \begin{align*}
    B_0(g, c^0, c^1) :=
    &\dom(g) \cap \Big\{\vartheta_g\big(\widehat{c^0}(n)\big) \, \Big| \, n \in \omega \land c^1(n) = 1 \land \widehat{c^0}(n) \in \dom(\vartheta_g) \Big\} \Big\backslash\\
    &\Big\{m \in \omega \, \Big| \, g(m) = e(g, c^0, c^1)(m) \Big\}.
    \end{align*}
    Recall that $\widehat{c^0}$ is the function enumerating $\range(c^0)$. Basically, what we are doing is that we are using $c^1$ to pick certain elements from $\range(c^0)$, and then using these picked elements as inputs of $\vartheta_g$. This way we have passed to a subset $B_0(g, c^0, c^1)$ of $D(g)$ in such a way that case (2) of Claim \ref{ad} (establishing a strong form of almost disjointness) holds.

    Above we have defined the notion of when a set is $g$-spaced. The notion is based on the notion of $(f, g)$-spacedness, taken from Subsection 1.2 of \cite{constructing_mcg}. Note that since $e(\cdot, \cdot, \cdot)$ respects the interval partition, our notion of $g$-spaced agrees with the notion of $(e(\chi(g), c^0, c^1), g)$-spaced from \cite{constructing_mcg} for any $c^0$ and $c^1$. Instead of saying that for each $g \in \injfininf$ and any $c^0, c^1 \in 2^{\leq \omega}$ (with $\lh(g) \leq \lh(c^0) = \lh(c^1)$) the set $B_0(g, c^0, c^1)$ is $g$-spaced, we rather simply say that $B_0$ is \emph{spaced}.

    Finally, observe that $D(g)$ and $B_0(g, c^0, c^1)$ are $\Delta^0_1(g)$ and $\Delta^0_1(g, c^0, c^1)$ respectively (in terms of that their membership is decidable using their respective oracles).

    \subsection{Relations coding extensions}

    We continue with a discussion of binary relations which will, by abstracting certain technical aspects, simplify some of the technical steps in the proof of the main theorem. The definitions of the relations
    $<^f_0$ and  $<^f_1$ are inspired by \cite{constructing_mcg}.

    Following \cite{compactness_med}, for $c \in 2^{\leq \omega}$ we say that it is \emph{good}, if for every two successive $n_0 < n_1$ from $c^{-1}[\{1\}]$ we have that
    \[
    n_1 \equiv \sum_{i \leq n_0} c(i) \cdot 2^i \pmod{2^{n_0 + 1}},
    \]
    and in case
    \[
        |\{n \in \omega \, | \, c(n) \text{ is defined and equals } 1\}| < \omega
    \]
    the sequence $c$ is finite. 
    
    Let $\mathcal C$ denote the set of all finite good sequences $c$, with either $c(-1) = 1$ or $c = \emptyset$. For $c, c' \in \mathcal{C}$, define the relation $c \trik c'$ by
    \[
    c \trik c' \quad :\Longleftrightarrow \quad \text{for some $n \in \omega$, } c' = c^\smallfrown 0^{n}1_2.
    \]
    Good sequences will be crucial for our construction (see the beginning of Subsection \ref{subsection_construction}) and for the nice properties used e.g. in Claim \ref{ad}.

    Define now a uniformly computable family
    \[
    \{<^f_0 \, | \, f \in (\omega)^{\leq \omega} \}
    \]
    of strict partial orders as follows: given $m, m' \in I_n$ define
    \[
        \delta_n(m, m') =
        \begin{cases}
    \parbox[t]{9cm}{the unique element $(w_0, w_1, w_2) \in W_n$ such that $e_n(w_0, w_1, w_2)(m) = m'$, if such exists; or}\\
    \uparrow \text{(remains undefined) otherwise}.
    \end{cases}
    \]
    For $f \in (\omega)^{\leq \omega}$ and $m \in \omega$, we let
    \[
    \delta(f, m) := \delta_n(m, f(m)) 
    \]
    for the unique $n$ such that $m \in I_n$, when $m \in \dom(f)$ and $\delta_n(m, f(m))$ is defined; otherwise we let $\delta(f, m)$ remain undefined. Finally, set
    \[
    m <^f_0 m'
    \]
    precisely when $m < m'$, $m' \in \dom(f)$, $\delta(f, m)$ and $\delta(f, m')$ are defined and
    \[
    \delta(f, m) = r^{m'}_m (\delta(f, m')).
    \]
    Of course, by ``uniformly computable'' we mean that there is some $a \in \omega$ such that the family is precisely the set 
    \[
    \big\{\{a\}^f \setdef f \in (\omega)^{\leq \omega}\big\},
    \]
    which is clearly the case.

    Next, we introduce another uniformly computable family of strict partial orders
    \[
        \{<^f_1 \, | \, f \in \injfininf \}
    \]
    as follows. For $m, m' \in \omega$ set $m <^f_1 m'$ if and only if it holds that $m < m'$, $f(m) < f(m')$ are both defined and there is some $g \in \injfin$ so that $\{f(m), f(m') \} \sq D(g)$. It is not hard to see that $<^f_1$ is a strict partial order (to prove transitivity use the nice properties of $\vartheta_g$). Note also that the family is uniformly computable,
    as the existence of the appropriate $g \in \injfin$ can be determined in finitely many steps.

    We now present an amplification of Claim 2.7 from \cite{compactness_med}, which will be essential for the proof of maximality of our constructed group (see Proposition \ref{prop_maximality}).

    \begin{lemma}\label{dichotomy}
        For any
        $g \in \injinf$
        one of the following holds:
        \begin{enumerate}
            \item There is some $I \in [D(g)]^{\infty}$, which is linearly ordered by $<^g_0$.
            \item There are good sequences $d^0, d^1 \in 2^\omega$, such that no two elements of $B_0(g, d^0, d^1)$ are $<^g_0$-comparable and either:
            \begin{enumerate}
                \item $B_0(g, d^0, d^1)$ is linearly ordered by $<^g_1$; or
                \item no two elements of $B_0(g, d^0, d^1)$ are $<^g_1$-comparable.
            \end{enumerate}
        \end{enumerate}
    \end{lemma}

    \begin{proof}
        Fix any
        $g \in \injinf$. Assume first that
        \begin{equation*}
            (\exists c_0  \in \mathcal{C})\, (\forall c_1  \in \mathcal{C}) \, c_0  \trik c_1  \to (\exists c_2  \in \mathcal{C})\, c_1  \trik c_2  \land \vartheta_g(\lh(c_1 ) - 1) <^g_0 \vartheta_g(\lh(c_2 )-1) \tag{$*$}
        \end{equation*}
        holds. Then fix $c_0 $ witnessing the first existential quantifier, let $c_1  \in \mathcal{C}$ be such that $c_0  \trik c_1 $ and set $n_0 := \vartheta_g(\lh(c_1 )-1)$. Using the second existential quantifier, we get some $c_2 $ with $c_1  \trik c_2 $ and $n_0 <^g_0 \vartheta_g(\lh(c_2 )-1)$. Put $n_1 := \vartheta_g(\lh(c_2 )-1)$ and let $d_2 \in \mathcal{C}$ be unique such that $\lh(d_2 ) = \lh(c_2 )$ and $c_0  \trik d_2 $. Applying $(*)$ with $d_2 $ instantiating the universal quantifier, we get $c_3 \in \mathcal{C}$ with $d_2  \trik c_3 $ and $n_1 <^g_0 \vartheta_g(\lh(c_3 )-1)$, so we set $n_2 := \vartheta_g(\lh(c_3 )-1)$. Clearly, we can proceed inductively to obtain an infinite sequence $(n_i)_{i \in \omega}$ contained in $D(g)$ and such that for every $i \in \omega$ it holds that $n_i <^g_0 n_{i+1}$. Then set
        \[
        I := \{n_i \, | \, i \in \omega\},
        \]
        to obtain (1).

        Suppose now that $(*)$ does not hold. Then its negation
        \[
        (\forall c_0  \in \mathcal{C})\, (\exists c_1  \in \mathcal{C})\, c_0  \trik c_1  \land (\forall c_2  \in \mathcal{C})\, c_1  \trik c_2  \to \neg \big(\vartheta_g(\lh(c_1 ) - 1) <^g_0 \vartheta_g(\lh(c_2 )-1)\big) \tag{$\neg *$}
        \]
        must be true. Now, first instantiate the leftmost universal quantifier of $(\neg *)$ with $c_0  := \emptyset$, to get $d_0  \in \mathcal{C}$ with $\emptyset \trik d_0 $. Next, instantiate the same quantifier with $c_0  := d_0 $, to get $d_1 \in \mathcal{C}$ with $d_0  \trik d_1 $. Inductively, we get a sequence $(d_i )_{i \in \omega}$, such that for every $i \in \omega$ it holds that $d _i \trik d _{i+1}$. With this, we define an infinite good sequence $d^0 \in 2^\omega$ by
        \[
        d^0 := \cup\{ d_i  \, | \, i \in \omega\}.
        \]
        
        Now we basically repeat the argument of case (1) to get the second infinite good sequence $d^1$. Assume first that
        \begin{equation*}
            (\exists c_0  \in \mathcal{C})\, (\forall c_1  \in \mathcal{C}) \, c_0  \trik c_1  \to (\exists c_2  \in \mathcal{C})\, c_1  \trik c_2  \land \vartheta_g(\widehat{d^0}(\lh(c_1 ) - 1)) <^g_1 \vartheta_g(\widehat{d^0}(\lh(c_2 )-1)) \tag{$**$}
        \end{equation*}
        holds. Then as in the case when $(*)$ is true, we can get a good sequence $d^1 \in 2^\omega$ such that the set
        \[
        I := \{ \vartheta_g(\widehat{d^0}(n)) \, | \, n \in \omega \land d^1(n) = 1 \}
        \]
        is linearly ordered by $<^g_1$. By definition it holds that $B_0(g, d^0, d^1) \sq I$, so $B_0(g, d^0, d^1)$ is linearly ordered by $<^g_1$ as well. Due to the property in the right part of $(\neg *)$, we have made sure that no two elements of $B_0(g, d^0, d^1)$ are $<^g_0$-comparable.
        We have thus obtained case (a) of (2).
        
        If on the other hand the negation of $(**)$,
        \begin{align}
        (\forall c_0  \in \mathcal{C})\, (\exists c_1  \in \mathcal{C})\, c_0  \trik c_1  \land &(\forall c_2  \in \mathcal{C})\, c_1  \trik c_2  \to \tag{$\neg {**}$}\\
        &\neg \big(\vartheta_g(\widehat{d^0}(\lh(c_1 ) - 1)) <^g_1 \vartheta_g(\widehat{d^0}(\lh(c_2 )-1))\big) \notag
        \end{align}
        holds, we repeat the procedure by which we have constructed $d^0$ to get an infinite good $d^1$. By definition of $B_0(g, d^0, d^1)$ and the right parts of $(\neg *)$ and $(\neg {**})$, we have made sure that no two elements of $B_0(g, d^0, d^1)$ are comparable with respect to either one of $<^g_0$ or $<^g_1$, obtaining case (b) of (2).
    \end{proof}

    \subsection{Further restriction of \texorpdfstring{$B_0$}{B0}} The definitions of $D$ and $B_0$ make sure that the sets $D(\cdot)$ and $B_0(\cdot, \cdot, \cdot)$ are mutually sparse,
    as the following claim shows (this is a stronger analogue to Claim 2.4 from \cite{compactness_med} and Lemma 1.9 from \cite{constructing_mcg}).

    \begin{claim}\label{ad}
        Suppose that $g, h \in \injinf$ and $c^0, c^1, d^0, d^1 \in 2^\omega$ are all good.
        \begin{enumerate}
            \item If $g \neq h$, then the set
        \[
        \{ n \in \omega \, | \, D(g) \cap I_n \neq \emptyset \land D(h) \cap I_n \neq \emptyset\}
        \]
        is finite.
        \item If $(g, c^0, c^1) \neq (h, d^0, d^1)$, then the set
        \[
        \{ n \in \omega \, | \, B_0(g, c^0, c^1) \cap I_n \neq \emptyset \land B_0(h, d^0, d^1) \cap I_n \neq \emptyset\}
        \]
        is finite.
        \end{enumerate}
    \end{claim}
    
    \begin{proof}
        For (1), let $m \in \omega$ be such that $g \res m \neq h \res m$. Then for all $k_0, k_1 > m$ it holds that $\vartheta_g(k_0)$ and $\vartheta_h(k_1)$ are not in the same interval $I_n$, and thus the proof of (1) is complete.

        Suppose for (2) that $g = h$ and that $c^j \neq d^j$ for some $j \in 2$. Let $m \in \omega$ be such that $c^j \res m \neq d^j \res m$. By definition of goodness, there is at most one $k \geq m$, for which it holds that both $c^j(k) = 1$ and $d^j(k) = 1$. Then a quick look at the definition of $B_0$  (left to the reader) completes the proof of (2).
    \end{proof}
    
    \begin{comment}
    \begin{claim}\label{ad}
        Suppose that $g, h \in S_\infty$ and $c, d \in 2^\omega$. If $g \neq h$ or $c \neq d$ are both good then there is some $n \in \omega$ so that for every $k_0, k_1 > n$ with $k_0 \in B_0(g, c)$ and $k_1 \in B_0(h, d)$ it holds that $k_0$ and $k_1$ are not in the same interval $I_m$.
    \end{claim}
        
    \begin{proof}
        In case $g \neq h$ let $n \in \omega$ be such that $g\res n \neq h \res n$ and in case $g = h$ and $c \neq d$ are both good let $n$ be such that $c \res n \neq d \res n$. Then for all $k_0, k_1 \geq n$, with $c(k_0) = 1$ and $d(k_1) =1$ it holds that
        \[
        2 \cdot 3^{\#(g \res (k_0 + 1))} \cdot 5^{\xi_g(k_0)} \neq 2 \cdot 3^{\#(h \res (k_1 + 1))} \cdot 5^{\xi_h(k_1)}.
        \]            
        By definition of $B_0(\cdot, \cdot)$, this completes the proof.
    \end{proof}
    \end{comment}

    Unfortunately, the sets $B_0(\cdot, \cdot, \cdot)$ are still not sparse enough for the construction to succeed, so we now restrict them further. 

     We will define a function
    \[
    B: \{(g, c^0, c^1) \in \injfininf \times 2^{\leq \omega} \times 2^{\leq \omega} \, | \, \lh(g) \leq \lh(c^0) = \lh(c^1) \} \to [\omega]^{\leq \omega},
    \]
    such that for every $(g, c^0, c^1) \in \dom(B)$ it holds that $B(g, c^0, c^1) \sq B_0(g, c^0, c^1)$.
    Moreover, we will require that the following property, called \emph{superspacedness} (since it is reminiscent of the notion of being \emph{spaced}, borrowed from \cite{constructing_mcg} and reintroduced above), holds for $B$ (note that since $B$ is pointwise contained in $B_0$, it is automatically spaced).

    \begin{definition}[Superspacedness]
        A function
        \[
        B: \{(g, c^0, c^1) \in \injfininf \times 2^{\leq \omega} \times 2^{\leq \omega} \, | \, \lh(g) \leq \lh(c^0) = \lh(c^1) \} \to [\omega]^{\leq \omega}
        \]
        is \emph{superspaced} if for every $g \in \injinf$, for which the set
        \[
        I := \{n \in D(g) \, | \, g(n) = e(\vec{x}, \vec{d^0}, \vec{d^1})(n)\}
        \]
        is infinite for some $(\vec{x}, \vec{d^0}, \vec{d^1}) \in \F(2^\omega \times 2^\omega \times 2^\omega)$, with $\vec{x} = x_{k}^{i_k} \cdots x_{0}^{i_0}$ and $\vec{d^j} = (d^j_k)^{i_k} \cdots (d^j_0)^{i_0}$ for $j \in 2$, defining 
        \[
        J := \{j \in k + 1 \, | \, x_j \in \range(\chi) \land \chi^{\dagger}(x_j) \neq g \land d^0_j, d^1_j \text{ are good}\},
        \]
        it holds that there are infinitely many $m \in I$, so that
        \[
        (\forall j \in J)\, \chi^{\dagger}(x_j)[B(\chi^{\dagger}(x_j), d^0_j, d^1_j)] \cap I_{n(m)} = \emptyset,
        \]
        where $n(m)$ is the unique such that $m \in I_{n(m)}$.
    \end{definition}

    The notion of superspacedness is similar to the notion of being \emph{cooperative} from \cite{constructing_mcg}. The purpose of the definition is the following: we will be catching elements $f \in \injinf$ on $D(f)$ by words in $\range(e)$, so we have to ensure that with the definition of $\dot{e}$ (using the definition of $e$ in certain cases; see Subsection \ref{subsection_construction}) we still catch the elements (now with words in $\range(\dot{e})$). This is formalised in Proposition \ref{prop_maximality}.

     We continue by defining $B$. We first define a set
     \[
     T \sq \injfin \times \{-1, 1\}^{<\omega} \times (2^{<\omega})^{<\omega} \times (2^{<\omega})^{<\omega} \times (2^{<\omega})^{<\omega},
     \]
     equipped with a tree like order $<$, and a function 
     \[
     \psi: T \to 2^{< \omega}.\]
     They will store information which we will then use to define $B$ from $B_0$. The process of defining $\psi$ will be algorithmic in nature and is similar to the algorithm called ``semaphore'' in \cite{constructing_mcg}, where it is also explained in detail (which the reader might find helpful). Let
	\begin{align*}
		T := \Big\{ (s, \vec{i}, \vec{x}, \vec{d^0}, \vec{d^1}) \in \injfin \times &\{-1, 1\}^{<\omega} \times (2^{<\omega})^{<\omega} \times (2^{<\omega})^{<\omega} \times (2^{<\omega})^{<\omega} \, \Big| \\
            &\lh(\vec{i}\, ) = \lh(\vec{x}) = \lh(\vec{d^0}) = \lh(\vec{d^1})\, \land\\
            &(\exists k \in \omega)\, \lh(s) = \sum_{m \leq k} |I_m|\, \land \, \\
            &(\forall j \in \lh(\Vec{x}))\, \lh(\Vec{x}(j)) = \lh(\vec{d^0}(j)) = \lh(\vec{d^1}(j)) = k
        \Big\}.
	\end{align*}
	Here, the first component $s$ is an approximation to an element of $\injinf$,
    and the last four components, $\vec{i}, \vec{x}, \vec{d^0}, \vec{d^1}$ determine the word
    \[
    w(\vec{i}, \vec{x}, \vec{d^0}, \vec{d^1}) \in \F(2^{\lh(\Vec{x}(0))} \times 2^{\lh(\Vec{d^0}(0))} \times 2^{\lh(\Vec{d^1}(0))}),
    \]
    defined by
    \[
    w(\vec{i}, \vec{x}, \vec{d^0}, \vec{d^1}) := \big(\vec{x}(-1), \vec{d^0}(-1), \vec{d^1}(-1)\big)^{\vec{i}(-1)} \cdots \big(\vec{x}(0), \vec{d^0}(0), \vec{d^1}(0)\big)^{\vec{i}(0)}.
    \]
    For $(s, \vec{i}, \vec{x}, \vec{d^0}, \vec{d^1}) \in T$ let $k(s)$ be the unique $k$ for which $\lh(s) = \sum_{m \leq k} |I_m|$.

    Clearly, $T$ is recursive. We next define a recursive tree-like ordering $<$ on $T$ as follows:
        \begin{align*}
		(s_0, \vec{i}_0, \vec{x}_0, \vec{d^0_0}, \vec{d^1_0}) < (&s_1, \vec{i}_1, \vec{x}_1, \vec{d^0_1}, \vec{d^1_1})\, :\Longleftrightarrow\\
  &s_0 \sqsubset s_1 \land \vec{i}_0 = \vec{i}_1 \land
            r^{k(s_1)}_{k(s_0)}\big(w(\vec{i}_1, \vec{x}_1, \vec{d^0_1}, \vec{d^1_1})\big) = w(\vec{i}_0, \vec{x}_0, \vec{d^0_0}, \vec{d^1_0}).
	\end{align*}
	Finally, we inductively define $\psi$. The idea is that for every $(f, p^0, p^1) \in \dom(B_0)$, we will look at the value of $\psi$ on the relevant elements of $T$ and based on the value of $\psi$ on those elements we will decide whether to remove elements from  $B_0(f, p^0, p^1)$ when defining $B(f, p^0, p^1)$ (see Equation \ref{def_B} below), so that the desired superspacedness condition holds (we will verify this in Claim \ref{superspaced_claim}).

     Let $t = (s, \vec{i}, \vec{x}, \vec{d^0}, \vec{d^1}) \in T$ and suppose we have defined $\psi$ on all $<$-predecessors of $t$ in $T$. Set $n := \lh(\Vec{i}\, )$ and for $j \in n$ define $g_j := \chi^\dagger(\Vec{x}(j))$. We will define $\psi(t)$, which will be an element of $2^n$. Let $t_* = (s_*, \vec{i}_*, \vec{x}_*, \vec{d^0_*}, \vec{d^1_*}) \in T$ be the strict predecessor of $t$ (note that when there is a strict predecessor it is unique; in case when $t$ does not have a strict predecessor, set
     \[
     t_* := (\emptyset, \Vec{i}, (\underbrace{\emptyset, \ldots, \emptyset}_{\lh(\vec{i}\, )}), (\underbrace{\emptyset, \ldots, \emptyset}_{\lh(\vec{i}\, )}), (\underbrace{\emptyset, \ldots, \emptyset}_{\lh(\vec{i}\, )}))
     \]
     and use 0 for any occurrence of $\psi(t_*)(j)$ below). For $j \in n$ set $g_j^* := \chi^\dagger(\Vec{x}_*(j))$. We now consider two cases.
     \begin{itemize}
         \item If there is some $j \in n$ for which
            \[
            \psi(t_*)(j)= 0
            \]
            and
            \begin{align*}
                (\exists m \leq k)\, (\exists l \in I_m)\, &l \in D(s) \, \land\\
                &\delta(s, l) = r^k_m \big(w(\vec{i}, \vec{x}, \vec{d^0}, \vec{d^1}) \big)\, \land \\
                &g_j[B_0(g_j, \vec{d^0}(j), \vec{d^1}(j))] \cap I_m \neq \emptyset\, \land\\
                &g_j^*[B_0(g_j^*, \vec{d^0_*}(j), \vec{d^1_*}(j))] \cap I_m = \emptyset,
            \end{align*}
            then for all such $j$ put
            \[
            \psi(t)(j) = 1
            \]
            and for all other $j' \in n$ keep 
            \[
            \psi(t)(j') = \psi(t_*)(j').
            \]
            \item  If there is no such $j$, put
            \[
            \psi(t)(j) = 0
            \]
            for all $j \in n$.
     \end{itemize}
     It is not hard to see that $\psi$ is recursive.

    Finally, we use $T$ and $\psi$ to define $B$. Take any $(f, p^0, p^1) \in \dom(B_0)$ and suppose that $m \in B_0(f, p^0, p^1)$.
    Then let %
    \begin{align}\label{def_B}
        m \notin B(f, p^0, p^1) \quad :\Longleftrightarrow \quad
        &(\exists (s, \vec{i}, \vec{x}, \vec{d^0}, \vec{d^1}) \in T)\, (\exists j \in \lh(\vec{i}\, ))\,  \\
        &\chi^{\dagger}(\vec{x}(j)) \sqsubseteq f \land \vec{d^0}(j) = p^0 \res k(s) \land \vec{d^1}(j) = p^1 \res k(s) \, \land\notag\\
        &m \in \dom(\chi^{\dagger}(\vec{x}(j))) \land
        (\exists l \in I_{n(f(m))})\, l \in D(s)\, \land \notag\\
        &\delta(s, l) = r^{k(s)}_{n(f(m))}\big( w(\vec{i}, \vec{x}, \vec{d^0}, \vec{d^1})\big) \, \land \notag\\
        &\psi(s, \vec{i}, \vec{x}, \vec{d^0}, \vec{d^1})(j) = 0\, \land \notag\\
        &(\forall (s_*, \vec{i}_*, \vec{x}_*, \vec{d^0_*}, \vec{d^1_*}) < (s, \vec{i}, \vec{x}, \vec{d^0}, \vec{d^1}))\, \notag\\
        &m \notin \dom(\chi^\dagger(\Vec{x}_*(j))) \lor k(s_*) < n(f(m)), \notag
    \end{align}
    where $n(f(m))$ is the unique $n$ such that $f(m) \in I_n$.
    Note that even though (\ref{def_B}) makes reference to an existential quantifier over $T$,
    the existence of the $(s, \vec{i}, \vec{x}, \vec{d^0}, \vec{d^1}) \in T$ satisfying the required properties can be determined in finitely many steps. This is because $D(\cdot)$ are mutually almost disjoint, and so
    the requirements on the right side of (\ref{def_B}) leave us only with finitely many potential candidates $(s, \vec{i}, \vec{x}, \vec{d^0}, \vec{d^1}) \in T$ which we need to consider. We leave the details to the reader.
    
    Hence, since $B_0$, $T$ and $\psi$ are all $\Delta^0_1$, $B$ is $\Delta^0_1$ as well. Moreover, it holds by design that $B(f, p^0, p^1) \sq B_0(f, p^0, p^1)$ for every $(f, p^0, p^1) \in \dom(B_0)$.

    The following claim identifies sufficient conditions for $B(f, p^0, p^1)$ to be infinite. This will be used later in Proposition \ref{prop_maximality}, where we prove that the constructed group is maximal.
    
    \begin{claim}\label{inf_B}
        Suppose that $f \in \injinf$ and $p^0, p^1 \in 2^ \omega$ are good so that $B_0(f, p^0, p^1)$ is infinite, no two elements of $B_0(f, p^0, p^1)$ are comparable with respect to $<^f_0$ and so that one of the following holds:
        \begin{enumerate}[(a)]
            \item $B_0(f, p^0, p^1)$ is linearly ordered by $<^f_1$; or
            \item no two elements of $B_0(f, p^0, p^1)$ are $<^f_1$-comparable.
        \end{enumerate}
        Then $B(f, p^0, p^1)$ is infinite.
    \end{claim}

    \begin{proof}
        In case (b) holds, it is easy to verify that $B(f, p^0, p^1) = B_0(f, p^0, p^1)$. So assume that (a) holds. The definition of $\psi$ makes sure that for every $m \in B_0(f, p^0, p^1) \setminus B(f, p^0, p^1)$, there is some $m' > m$ with $m' \in B(f, p^0, p^1)$. In particular, $B(f, p^0, p^1)$ is infinite. The details are left to the reader.
    \end{proof}

    Finally, we prove that $B$ is superspaced.

    \begin{claim}\label{superspaced_claim}
        Suppose that for $g \in \injinf$, $\vec{x} = x_{k}^{i_k} \cdots x_{0}^{i_0} \in \F(2^\omega)$ and $\vec{d^j} = (d^j_k)^{i_k} \cdots (d^j_0)^{i_0} \in \F(2^\omega)$ for $j \in 2$ it holds that the set
    \[
    I := \{n \in D(g) \, | \, g(n) = e(\vec{x}, \vec{d^0}, \vec{d^1})(n)\}
    \]
    is infinite. Then, setting
    \[
        J := \{j \in k + 1 \, | \, x_j \in \range(\chi) \land \chi^{\dagger}(x_j) \neq g \land d^0_j, d^1_j \text{ are good}\},
    \]
    there are infinitely many $m \in I$ for which
    \[
    (\forall j \in J)\, \chi^{\dagger}(x_j)[B(\chi^{\dagger}(x_j), d^0_j, d^1_j)] \cap I_{n(m)} = \emptyset.
    \]
    \end{claim}

    \begin{proof}
        Set $g_j := \chi^{\dagger}(x_j)$ for every $j \in k+1$. Then let
        \[
        J_0 := \{j \in J \, | \, (\exists^\infty m \in I)\, g_j[B(g_j, d^0_j, d^1_j)] \cap I_{n(m)} \neq \emptyset \}.
        \]
        Let $m_0 \in \omega$ be such that for every $m \in I \setminus m_0$ it holds that
        \[
        (\forall j \in J \setminus J_0)\, g_j[B(g_j, d^0_j, d^1_j)] \cap I_{n(m)} = \emptyset.
        \]
        Let $\vec{i} := (i_0, \ldots, i_k)$ and for $n \in \omega$ use $\vec{x} \res n$ to denote $(x_0 \res n, \ldots, x_k \res n)$, $\vec{d^0} \res n$ to denote $(d^0_0 \res n, \ldots, d^0_k \res n)$ and  $\vec{d^1} \res n$ to denote $(d^1_0 \res n, \ldots, d^1_k \res n)$. For $l \in \omega$, we also let $N(l) := \min I_{n(l) + 1}$ (recall that $n(l)$ is the unique $n$ such that $l \in I_n$). By case analysis of the definition of $\psi$ on
        \[
        \big(g \res N(l), \vec{i}, \vec{x} \res n(l), \vec{d^0} \res n(l), \vec{d^1} \res n(l)\big),
        \]
        one can see from the definition of $B$ that there are infinitely many $m \in I \setminus m_0$, so that for every $j \in J_0$ it holds that
        \[
         g_j[B(g_j, d^0_j, d^1_j)] \cap I_{n(m)} = \emptyset.
        \]
        With this the proof is complete.
    \end{proof}
    
    \subsection{The construction}\label{subsection_construction}

    We finally begin the construction. Define $\dot{e}: 2^\omega \times 2^\omega \times 2^\omega \to S_\infty$ on $(x, c^0, c^1) \in 2^\omega \times 2^\omega \times 2^\omega$ 
    as follows. With $g := \chi^{\dagger}(x) \in \injfininf$ let
    \begin{equation*}
    \Dot{e}(x, c^0, c^1)(n) := \begin{cases}
    g(n) & \parbox[t]{9cm}{if $n \in B(g, c^0, c^1)$, $c^0 \res (n + 1)$, $c^1 \res (n + 1)$ are good and $\neg (\exists n_0, n_1 \in B_0(g, c^0, c^1) \cap n)\, n_0 <^g_0 n_1$;}\\
    e(x, c^0, c^1)(m) & \parbox[t]{9cm}{if $m := g^{-1}(n)$ is defined and every condition from case 1 is satisfied when we replace $n$ with $m$;}\\
    e(x, c^0, c^1)^2(n) & \parbox[t]{9cm}{if $m := (g^{-1} \circ e(x, c^0, c^1))(n)$ is defined and every condition from case 1 is satisfied when we replace $n$ with $m$;}\\
    e(x, c^0, c^1)(n) & \text{otherwise.}
    \end{cases}
    \end{equation*}

    \begin{fact}\label{f.spaced.used1}
    The map $\Dot{e}$ is well-defined and $\Dot{e}(x, c^0, c^1)$ is a permutation for every $(x, c^0, c^1) \in 2^\omega \times 2^\omega \times 2^\omega$.
    \end{fact}
    \begin{proof}
    Note that the cases in the above definition are mutually exclusive, thanks to the definition of $D(g)$ (we are still using $g := \chi^{\dagger}(x)$), which makes sure that $D(g)$ and hence also $B(g, c^0, c^1)$ are spaced. When one of the first three cases takes place, the other two of the first three cases make sure that $\dot{e}(x, c^0, c^1)$ is a permutation. When this happens, we say that we are performing a surgery operation (since we are surgically joining two orbits; see Figure \ref{fig_surgery}). The details are left to the reader.
    \end{proof}
    
    Note that in case $x \notin \range(\chi)$ it holds that $g = \chi^{\dagger}(x) \in \injfin$. In this case, we only perform finitely many surgeries and afterwards only use $e$. The definition of $\dot{e}$ is elucidated in Figure \ref{fig_surgery} below (adapted after Figure 1 from \cite[p.~8]{constructing_mcg}), where the orbit structure is shown before (in black) and after (in wine red) the surgery operation.

    \begin{figure}[ht]
    \begin{tikzcd}
    \dots & \arrow[l, "{e(x, c^0, c^1)}"] \arrow[l, rdeca, bend right, dashed]
    ({e(x, c^0, c^1)}\circ g)(n)    &  \arrow[l, "{e(x, c^0, c^1)}"]    g(n) \arrow[dl, rdeca, bend left, dashed] &  \arrow[l, "{e(x, c^0, c^1)}"] \arrow[ll, rdeca, bend right, dashed, "{\dot{e}(x, c^0, c^1)}"] ({e(x, c^0, c^1)}^{-1}\circ g)(n) & 
    \arrow[l, "{e(x, c^0, c^1)}"] \arrow[l, rdeca, bend right, dashed] \dots \\
    \dots & \arrow[l, "{e(x, c^0, c^1)}"] \arrow[l, rdeca, bend right, dashed]
    e(x, c^0, c^1)(n)   &  \arrow[l, "{e(x, c^0, c^1)}"] \arrow[u, modra, "g"]  \arrow[u, rdeca, bend right, dashed]   n   &  \arrow[l, "{e(x, c^0, c^1)}"]  \arrow[l, rdeca, bend right, dashed] e(x, c^0, c^1)^{-1}(n) & 
    \arrow[l, "{e(x, c^0, c^1)}"] \arrow[l, rdeca, bend right, dashed] \dots
    \end{tikzcd}

    \caption{The definition of $\dot{e}(x, c^0, c^1)$, where $g = \chi^{\dagger}(x)$, $n \in B(g, c^0, c^1)$, $\neg (\exists n_0, n_1 \in B_0(g, c^0, c^1) \cap n)\, n_0 <^g_0 n_1$ and $c^0 \res n$, $c^1 \res n$ are both good.}\label{fig_surgery}
    \end{figure}
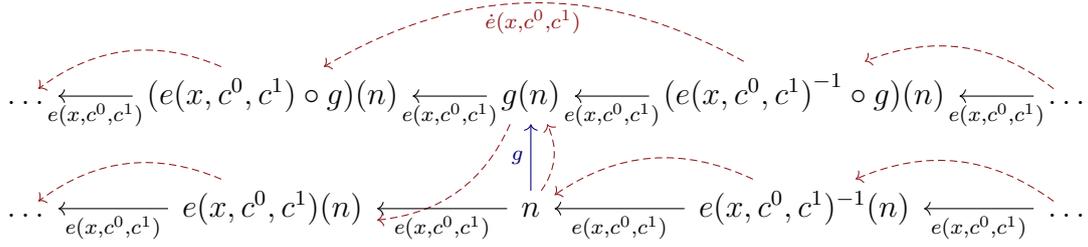
    
    The following is a reformulation of Theorem \ref{sigma_mcg}.
    
    \begin{theorem}\label{constructed_mcg}
        The set $\range(\dot{e})$ is a $\Pi^0_1$ subset of $\baire$
        and freely generates a $\Sigma^0_2$ mcg.
    \end{theorem}
    
    Set $\G := \langle \range(\dot{e}) \rangle \leq S_\infty$. As in \cite{constructing_mcg}, we split the long proof into propositions with shorter proofs. We first prove that $\G$ satisfies a strong form of maximality. The idea is based on the proof of Proposition 1.13 from \cite{constructing_mcg}.
        
    \begin{proposition}\label{prop_maximality}
        For every $g \in \injinf$ there is some $h \in \G$ such that
        \[
        \{n \in \omega \, | \, g(n) = h(n)\}
        \]
        is infinite.
    \end{proposition}
        
        \begin{proof}
            If $g$ has infinitely many fixed points, then clearly $\id_\omega \in \G$ agrees with $g$ on infinitely many places. So let $m \in \omega$ be such that for every $n \geq m$ it holds that $g(n) \neq n$. Applying Lemma \ref{dichotomy} to $g$, we consider the following two cases:
            
                \noindent
                (1) We have some $I \in [D(g)]^{\infty}$ which is linearly ordered by $<^g_0$.
                By definition of $<^g_0$, there is some $(\vec{x}, \vec{d^0}, \vec{d^1}) \in \F(2^\omega \times 2^\omega \times 2^\omega)$, with
                \[
                \vec{x} = x_k^{i_k} \ldots x_0^{i_0} \quad \text{and} \quad \vec{d^j} = (d^j_k)^{i_k} \ldots (d^j_0)^{i_0}
                \]
                for $j \in 2$, so that
                \[
                e(\vec{x}, \vec{d^0}, \vec{d^1}) \res I = g \res I.
                \]
                For every $j \in k + 1$ set also $g_j := \chi^{\dagger}(x_j) \in \injfininf$. If there is some $j \in k+1$, for which $g_j = g$, $d^0_j, d^1_j$ are good, no two elements of $B_0(g, d^0_j, d^1_j)$ are $<^g_0$-comparable and $B(g, d^0_j, d^1_j)$ is infinite, it follows that
                \[
                    \dot{e}(g, d^0_j, d^1_j) \res B(g, d^0_j, d^1_j) =  g \res B(g, d^0_j, d^1_j),
                \]
                so we are done. Otherwise define the set $I' \supseteq I$ by
                \[
                I' := \{n \in D(g) \, | \, g(n) = e(\vec{x}, \vec{d^0}, \vec{d^1})(n)\}
                \]
                and the sets
                \begin{align*}
                &J_0 := \{ j \in k+1 \, | \, g_j = g \land d^0_j, d^1_j \text{ are good } \land (\exists n_0, n_1 \in B_0(g, d^0_j, d^1_j))\, n_0 <^g_0 n_1\}\\
                &J_1 := \{j \in k+1 \, | \, x_j \in \range(\chi) \land (d^0_j \text{ is not good} \lor d^1_j \text{ is not good})\}\\
                &J_2 := \{j \in k+1 \, | \, x_j \notin \range(\chi) \lor B(g_j, d^0_j, d^1_j) \text{ is finite} \}\\
                &J_3 := \{j \in k + 1 \, | \, x_j \in \range(\chi) \land g_j \neq g \land d^0_j, d^1_j \text{ are good}\}.
                \end{align*}
                Let $m_0 \in \omega$ be such that for every $j \in J_0$ it holds that there are some $n_0, n_1 \in B_0(g, d^0_j, d^1_j) \cap m_0$ for which $n_0 <^g_0 n_1$. Next, let $m_1 \in \omega$ be such that for every $j \in J_1$, it holds that either $d^0_j \res m_1$ is not good or $d^1_j \res m_1$ is not good. Finally, let $m_2 \in \omega$ be such that for every $j \in J_2$ it holds that either $\lh(g_j) < m_2$ or $B(g, d^0_j, d^1_j) \sq m_2$. Then define $m_3 \in \omega$ to be the smallest element of $\{\min I_n \, | \, n \in \omega\}$ such that for every element $q$ of 
                \[
                \bigcup \big\{\{l, g_j(l), g_j^{-1}(l)\} \, \big| \, l \leq \max\{m_0, m_1, m_2\} \land j \in J_0 \cup J_1 \cup J_2\big\}
                \]
                it holds that $q < m_3$. Note that some of the $g_j(l)$ and $g_j^{-1}(l)$ above might be undefined, in which case we ignore them.

                By the superspacedness property of $B$ it holds that there is some infinite $I'' \sq I' \setminus m_3$ so that for every $m \in I''$ it holds that
                \[
                (\forall j \in J_3)\, g_j[B(g_j, d^0_j, d^1_j)] \cap I_{n(m)} = \emptyset,
                \]
                where $n(m)$ is the unique such that $m \in I_{n(m)}$.

                If $j \in J_0 \cup J_1 \cup J_2$, then
                \[
                    \dot{e}(x_j, d^0_j, d^1_j) \res (\omega \setminus m_3) = e(x_j, d^0_j, d^1_j) \res (\omega \setminus m_3).
                \]
                Finally, if $j \in J_3$, the definition of $\dot{e}$ ensures that for every $m \in I''$ and every $m' \in I_{n(m)}$ it holds that
                \[
                \dot{e}(x_j, d^0_j, d^1_j)(m') = e(x_j, d^0_j, d^1_j)(m').
                \]
                Thus we have shown that
                \[
                \dot{e}(\vec{x}, \vec{d^0}, \vec{d^1}) \res I'' = e(\vec{x}, \vec{d^0}, \vec{d^1}) \res I'' = g \res I''.
                \]

                \noindent
                (2) There are good $d^0, d^1 \in 2^\omega$ such that no two elements of $B_0(g, d^0, d^1)$ are $<^g_0$-comparable and
                \begin{enumerate}[(a)]
                    \item $B_0(g, d^0, d^1)$ is linearly ordered by $<^g_1$; or
                    \item no two elements of $B_0(g, d^0, d^1)$ are comparable with respect to $<^g_1$.
                \end{enumerate}
                If $B_0(g, d^0, d^1)$ is infinite, then so is $B(g, d^0, d^1)$ by Claim \ref{inf_B}, so
                we get by definition that
                \[
                \Dot{e}(\chi(g), d^0, d^1) \res B(g, d^0, d^1) = g \res B(g, d^0, d^1).
                \]
                If on the other hand $B_0(g, d^0, d^1)$ is finite,
                the definition of $B_0(g, d^0, d^1)$ implies that
                \[
                (\exists^\infty m \in \omega)\, e(\chi(g), d^0, d^1)(m) = g(m).
                \]
                For every such $m$ it holds by the definition of $\dot{e}$ that
                \[
                \dot{e}(\chi(g), d^0, d^1)(m) = e(\chi(g), d^0, d^1)(m).
                \]
            
            Since in both cases (1) and (2) we got an element of $\G$, agreeing with $g$ on an infinite set, the proof is complete.
        \end{proof}

        Next we move to the proof of cofinitariness, which is a bit more involved. The proof described below closely follows the proof of Proposition 1.14 from \cite{constructing_mcg}.
        
        \begin{proposition}\label{G_cofinitary}
            $\G$ is cofinitary.
        \end{proposition}
        
        \begin{proof}
            Let
            \[
            c = (c_l)^{i_l} \circ \cdots \circ (c_0)^{i_0},
            \]
            for $\{ c_j \}_{j \in l + 1} \sq \range(\dot{e})$ and $\{i_j\}_{j \in l + 1} \sq \{-1, 1\}$, be some reduced word in $\range(\dot{e})$, which has infinitely many fixed points. For $j \in l + 1$ let $x_j \in 2^\omega$ and $d^0_j, d^1_j \in 2^\omega$ be such that $\dot{e}(x_j, d^0_j, d^1_j) = c_j$, and put $g_j := \chi^{\dagger}(x_j) \in \injfininf$. Note that
            \[
            w:= (x_l, d^0_l, d^1_l)^{i_l} \cdots (x_0, d^0_0, d^1_0)^{i_0}
            \]
            is a reduced word in $\F(2^\omega \times 2^\omega \times 2^\omega)$.
            
            Let $F \sq \fix(c)$ be a tail segment, for which it holds that for every $m \in F$ and every $m'$ in the path of $m$ under $c$, $m'$ lies in the same interval $I_k$ with at most one $B(g_j, d^0_j, d^1_j)$, for which $d^0_j \res (\min I_{k+1}), d^1_j \res (\min I_{k+1})$ are good. This is possible by Claim \ref{ad}.
        
            For every $m \in F$ there is some $l(m) \in \omega$, so that
            \[
            c(m) = (a^m_{l(m)}  \circ \cdots \circ a^m_0)(m),
            \]
            where for $k \in l(m) + 1$, each $a^m_k$ is either $(g^m_k)^{i^m_k}$ or $e(x^m_k, (d^m_k)^0, (d^m_k)^1)^{i^m_k}$, where $i^m_k \in \{-1, 1\}$ and $g^m_k = g_j$ or $(x^m_k, (d^m_k)^0, (d^m_k)^1) = (x_j, d^0_j, d^1_j)$ for some $j \in l + 1$. This holds by unfolding the definition of $\dot{e}$ (and is left to the reader). Write also
            \[
            w^m := a^m_{l(m)} \cdots  a^m_0.
            \]
            Note that $l(m) \leq 2 l + 2$ by definition of $\dot{e}$. We can write $F$ as a finite union of sets, on each of which $w^{(\cdot)}$ is constant. Let $F^*$ be one of these sets, which is infinite. We replace every superscript $m$ with $\ast$, so that we have $l^* = l(m)$, $w^* = w^m$ and $a^*_k = a^m_k$, where $a^*_k$ is now either $(g^*_k)^{i^*_k}$ or $e(x^*_k, (d^*_k)^0, (d^*_k)^1 )^{i^*_k}$, for all $k \in l^* + 1$.

            \begin{claim}\label{w_reduces_to_empty}
                The word $w^*$ reduces to the empty word in $\F(\{(a^*_k)^{i^*_k} \, | \, k \in l^* + 1 \})$.
            \end{claim}

            We consider elements $(a^*_k)^{i^*_k}$, which are either $g^*_k$ or $e(x^*_k, (d^*_k)^0, (d^*_k)^1)$, as abstract generators in the above statement (we used $(a^*_k)^{i^*_k}$ and not $a^*_k$, so that the abstract generators $g^*_k$ or $e(x^*_k, (d^*_k)^0, (d^*_k)^1)$ are without powers).
            
            \begin{cproof}[Proof of Claim.]
                Suppose for contradiction that when seen as an abstract element of $\F(\{(a^*_k)^{i^*_k} \, | \, k \in l^* + 1 \})$, the word $w^*$ reduces to $v \neq \emptyset$. Then there is some $r \in \omega$ and a sequence $k(0), k(1), \ldots, k(r)$, for which
                \[
                v = a^*_{k(r)} \cdots a^*_{k(0)}.
                \]
                Define for $m \in F^*$ the sequence $(m(u))_{u = 0}^r$ inductively by $m(0) := m$, and for $u < r$ by
                \[
                m(u + 1) := a^*_{k(u)}(m(u)).
                \]
                We can assume without loss of generality that for any proper subword
                \[
                a^*_{k(r_1)} \cdots a^*_{k(r_0)}
                \]
                of $v$, where $0 \leq r_0 < r_1 \leq r$, it holds that $m(r_0) \neq m(r_1)$. Indeed, if this is not already the case, we use the Pigeonhole principle to find an infinite subset $F'$ of $F^*$ and a subword
                \[v' := a^*_{k(r_1)} \cdots a^*_{k(r_0)}
                \]
                of $w$, so that for every $m \in F'$ it holds that $m(r_0) = m(r_1)$ and so that for any proper subword
                \[
                v'' := a^*_{k(r_3)} \cdots a^*_{k(r_2)}
                \]
                of $v'$, with $r_0 \leq r_2 < r_3 \leq r_1$ it holds that $m(r_2) \neq m(r_3)$.
                
                Define
                \[
                n(m) := \min \{n \in \omega \, | \, (\exists u \in r + 1)\, m(u) \in I_n\},
                \]
                i.e., $n(m)$ is the least index of an interval which we pass through with $(m(u))_{u = 0}^r$. Note that this is slightly different from how we previously defined $n(m)$.
                
                \begin{subclaim}\label{subclaim_lowest_interval}
                    For all $u \in r + 1$ it holds that $m(u) \in I_{n(m)}$.
                \end{subclaim}
                
                \begin{cproof}[Proof of Subclaim.]
                    Suppose not. Then at least one of the following happens:
                    \begin{enumerate}
                        \item there is $0 < u_0 < r$ so that $m(u_0 - 1) \notin I_{n(m)}$ and $m(u_0) \in I_{n(m)}$; and $u_0 < u_1 < r$ so that $m(u_1) \in I_{n(m)}$ and $m(u_1 + 1) \notin I_{n(m)}$; or
                        \item there is $0 < u_0 < r$ so that $m(u_0 - 1) \in I_{n(m)}$ and $m(u_0) \notin I_{n(m)}$; and $u_0 < u_1 < r$ so that $m(u_1) \notin I_{n(m)}$ and $m(u_1 + 1) \in I_{n(m)}$.
                    \end{enumerate}

                    We can assume without loss of generality that case (1) happens. Note that it must hold that $a^*_{k(u_0-1)}$ is equal to $(g^*_j)^{-1}$ for some $j \in l + 1$ and $a^*_{k(u_1)}$ is equal to $g^*_j$. This is by property (iii) of $\vartheta_g$, the definition of $\Dot{e}$, the fact that $e$ stays inside the intervals and the definition of $F$. But
                    since $B(g^*_j, (d^*_j)^0, (d^*_j)^1)$ intersects $I_{n(m)}$ in exactly one point, it must be that $m(u_0 - 1) = m(u_1 + 1)$. This contradicts our assumption on $v$.
                \end{cproof}

                \begin{subclaim}
                    For at most one $u \in r + 1$ is it the case that $a^*_{k(u)}$ equals $g^*_j$ or $(g^*_j)^{-1}$ for some $j \in l + 1$.
                \end{subclaim}
                
                \begin{cproof}[Proof of Subclaim.]
                    This holds by the definition of $F$ (which used Claim \ref{ad}), by the fact that $B(g^*_j, (d^*_j)^0, (d^*_j)^1)$ intersects each $I_n$ in at most one point, the definition of $\Dot{e}$ and by the assumption that $v$ has no subwords with infinitely many fixed points on the paths starting with elements of $F^*$. We leave it to the reader to provide the details.
                \end{cproof}
                
                \begin{subclaim}
                    It cannot be the case that any $a^*_{k(u)}$ equals $g^*_j$ or $(g^*_j)^{-1}$ for any $j \in l + 1$.
                \end{subclaim}
                
                \begin{cproof}[Proof of Subclaim.]
                    By the previous subclaim we know that there is at most one $a^*_{k(u)}$ which equals $g^*_j$ or $(g^*_j)^{-1}$. For contradiction suppose that $v$ is of the form (the argument with $(g^*_j)^{-1}$ is analogous):
                    \begin{multline*}
                    e\big(x^*_{k(r)}, (d^*_{k(r)})^0, (d^*_{k(r)})^1\big)^{i^*_{k(r)}} \cdots \\
                    \cdots e\big(x^*_{k(\Bar{r} + 1)}, (d^*_{k(\Bar{r} + 1)})^0, (d^*_{k(\Bar{r} + 1)})^1\big)^{i^*_{k(\bar{r} + 1)}}\, g^*_j \,
                    e\big(x^*_{k(\Bar{r} - 1)}, (d^*_{k(\Bar{r} - 1)})^0, (d^*_{k(\Bar{r} - 1)})^1\big)^{i^*_{k(\bar{r} - 1)}} \cdots\\
                    \cdots e\big(x^*_{k(0)}, (d^*_{k(0)})^0, (d^*_{k(0)})^1\big)^{i^*_{k(0)}}
                    \end{multline*}
                    for some $\bar{r} \in r + 1$. But then it holds for infinitely many $n \in \omega$ that
                    \begin{multline}\label{g_caught}
                       g^*_j(n) =  \Big(e\big(x^*_{k(\Bar{r} + 1)}, (d^*_{k(\Bar{r} + 1)})^0, (d^*_{k(\Bar{r} + 1)})^1\big)^{-i^*_{k(\bar{r} + 1)}} \cdots\\
                       \cdots e\big(x^*_{k(r)}, (d^*_{k(r)})^0, (d^*_{k(r)})^1\big)^{-i^*_{k(r)}} \, e\big(x^*_{k(0)}, (d^*_{k(0)})^0, (d^*_{k(0)})^1\big)^{- i^*_{k(0)}}
                    \cdots \\
                   \cdots e\big(x^*_{k(\Bar{r} - 1)}, (d^*_{k(\Bar{r} - 1)})^0, (d^*_{k(\Bar{r} - 1)})^1 \big)^{-i^*_{k(\bar{r} - 1)}}\Big)(n).  
                    \end{multline}
                    Let $I$ be the infinite set of all such $n \in \omega$. Since also
                    \[
                    g^*_j \res I = \Dot{e}(\chi(g^*_j), (d^*_j)^0, (d^*_j)^1) \res I,
                    \]
                    (this is the only way in which $a^*_{k(u)}$ was able to be $g^*_j$), the definition of $\Dot{e}$ (see the first case) implies that elements of $I$ are pairwise $<^{g^*_j}_0$-incomparable. On the other hand, equation (\ref{g_caught}) implies that $I$ forms a $<^{g^*_j}_0$-increasing chain. This is a contradiction.
                \end{cproof}
                
                \begin{subclaim}
                    It is the case that $v = \emptyset$.
                \end{subclaim}
                
                \begin{cproof}[Proof of Subclaim.]
                    By the previous subclaim it holds that
                    \[
                    v = e\big(x^*_{k(r)}, (d^*_{k(r)})^0, (d^*_{k(r)})^1\big)^{i^*_{k(r)}} \cdots e\big(x^*_{k(0)}, (d^*_{k(0)})^0, (d^*_{k(0)})^1\big)^{i^*_{k(0)}}.
                    \]
                    But since $v$ has infinitely many fixed points and the range of $e$ generates a cofinitary subgroup it must hold that $v = \emptyset$.
                \end{cproof}
                
                This is a contradiction with our assumption that $v \neq \emptyset$, and with this the proof of Claim \ref{w_reduces_to_empty} is complete.
            \end{cproof}
            
            \begin{claim}\label{empty_word}
                It must be the case that $w^*$ is the empty word.
            \end{claim}
            
            \begin{cproof}[Proof of Claim.]
                Suppose for contradiction that $w^*$ is not the empty word. By considering all possible cases we will conclude that this is a contradiction. Suppose first that $w^*$ contains a subword of the form
                \[
                (g^*_k)^{-1} \  g^*_k
                \]
                for some $k \in l^*  +1 $. By definition of $\dot{e}$, this subword can only arise from a subword of $w$ (recall that $w$ is reduced), in two ways. The first option is that $(g^*_k)^{-1} \  g^*_k$ arose from
                \[
                (x^*_k, (d^*_k)^0, (d^*_k)^1) \ (x^*_k, (d^*_k)^0, (d^*_k)^1),
                \]
                and where the corresponding right $\dot{e}(x^*_k, (d^*_k)^0, (d^*_k)^1)$ was substituted by $g^*_k$ and the left $\dot{e}(x^*_k, (d^*_k)^0, (d^*_k)^1)$ by $e(x^*_k, (d^*_k)^0, (d^*_k)^1) \ (g^*_k)^{-1}$, so that we actually have the following subword of $w^*$
                \[
                e(x^*_k, (d^*_k)^0, (d^*_k)^1) \ (g^*_k)^{-1} \ g^*_k.
                \]
                But as we have observed that $w^*$ reduces to $\emptyset$, the occurrence of $e(x^*_k, (d^*_k)^0, (d^*_k)^1)$ must cancel out, so actually, there must be a subword of $w^*$ of the form
                \[
                e(x^*_k, (d^*_k)^0, (d^*_k)^1)^{-1}\ e(x^*_k, (d^*_k)^0, (d^*_k)^1) \ (g^*_k)^{-1} \ g^*_k
                \]
                or of the form
                \[
                e(x^*_k, (d^*_k)^0, (d^*_k)^1) \ (g^*_k)^{-1} \ g^*_k \ e(x^*_k, (d^*_k)^0, (d^*_k)^1)^{-1}.
                \]
                In both cases the occurrence of $e(x^*_k, (d^*_k)^0, (d^*_k)^1)^{-1}$ arose by substitution of
                \[
                    (x^*_k, (d^*_k)^0, (d^*_k)^1)^{-1}.
                \]
                This is a contradiction, as this means that $w$ contains a subword of the form
                \[
                (x^*_k, (d^*_k)^0, (d^*_k)^1)^{-1}\ (x^*_k, (d^*_k)^0, (d^*_k)^1)\  (x^*_k, (d^*_k)^0, (d^*_k)^1)
                \]
                or of the form
                \[
                (x^*_k, (d^*_k)^0, (d^*_k)^1)\ (x^*_k, (d^*_k)^0, (d^*_k)^1)\ (x^*_k, (d^*_k)^0, (d^*_k)^1)^{-1},
                \]
                implying that $w$ is not reduced. The second option is that $(g^*_k)^{-1} \  g^*_k$ arose from
                \[
                (x^*_k, (d^*_k)^0, (d^*_k)^1)^{-1} \ (x^*_k, (d^*_k)^0, (d^*_k)^1)^{-1},
                \]
                and where the corresponding left $\dot{e}(x^*_k, (d^*_k)^0, (d^*_k)^1)^{-1}$ was substituted by $(g^*_k)^{-1}$ and the right $\dot{e}(x^*_k, (d^*_k)^0, (d^*_k)^1)^{-1}$ by $g^*_k \ e(x^*_k, (d^*_k)^0, (d^*_k)^1)^{-1}$, so that we actually have the following subword of $w^*$
                \[
                (g^*_k)^{-1} \ g^*_k \  e(x^*_k, (d^*_k)^0, (d^*_k)^1)^{-1}.
                \]
                A contradiction is then established in a similar manner as in the first case.
                
                Next, suppose that $w^*$ contains a subword of the form
                \[
                e(x^*_k, (d^*_k)^0, (d^*_k)^1)^{-1} \ e(x^*_k, (d^*_k)^0, (d^*_k)^1).
                \]
                The only way this subword can appear in $w^*$ is (by applying the definition of $\dot{e}$) that $(x^*_k, (d^*_k)^0, (d^*_k)^1)^{-1}\ (x^*_k, (d^*_k)^0, (d^*_k)^1)$ was a subword of $w$. This is clearly a contradiction.
                
                By a similar argument, one can easily prove that $w^*$ cannot have subwords of the form $g^*_k \ (g^*_k)^{-1}$ or $e(x^*_k, (d^*_k)^0, (d^*_k)^1) \ e(x^*_k, (d^*_k)^0, (d^*_k)^1)^{-1}$, hence the proof is complete.
            \end{cproof}
            
            Claim \ref{empty_word} implies that $w = \emptyset$ and so $c = \id_\omega$. Since $c$ was an arbitrary element of $\G$ with infinitely many fixed points, the proof of Proposition \ref{G_cofinitary} is complete.
        \end{proof}

        Before we prove the next proposition, we introduce the following useful notion. For $\bar{f} \in \injfin$, we say that $\bar{f}$ is of \emph{interval length} $k$, if $\lh(\bar{f}) = \sum_{m \leq k} |I_m|$. For $\bar{f}$ of interval length $k$, we say that $(\bar{x}, \bar{d^0}, \bar{d^1}) \in 2^k \times 2^k \times 2^k$
        is \emph{recovered from} $\bar{f}$, if for every $m \leq k$ it holds that
        \[
            |\{n \in I_m \, | \, e(\bar{x}, \bar{d^0}, \bar{d^1})(n) \neq \bar{f}(n) \}| \leq 3.
        \]
        Next, for $\bar{x} \in 2^{<\omega}$
        and $\bar{g} \in (\omega)^{\leq \lh(\bar{x})}$, we say that $\bar{g}$ is $\bar{x}$-\emph{compatible}, if $\chi^{\dagger}(\bar{x}) \sqsubseteq \bar{g}$ and
        \begin{enumerate}[(I)]
            \item if $\bar{x}$ is of the form
            \[
            (\ldots, 1, \underbrace{0, 0, \ldots, 0}_n, 1, \ldots, 1, \underbrace{0, 0, \ldots, 0}_n, 1, \ldots)
            \]
            for some $n \in \omega$ (and is thus not in $\range(\chi)$ as it violates injectivity), then
            \[
            \bar{g} = \chi^{\dagger}(\bar{x});
            \]
            \item\label{i.copmatible.geq} if there is some $n \in \omega$ such that
            \[
            \bar{x} = \chi(\chi^{\dagger}(\bar{x})) ^\smallfrown (\underbrace{0, \ldots, 0}_n),
            \]
            i.e., $\bar{x}$ is not in $\range(\chi)$ because of the second reason described below the definition of $\chi^{\dagger}$, namely that it does not end with a 1, then in case $\lh(\bar{g}) > \lh(\chi^{\dagger}(\bar{x}))$, it must hold that
            \[
            \bar{g}(\lh(\chi^{\dagger}(\bar{x}))) \geq n;
            \]
            \item otherwise (when $\bar{x} \in \range(\chi)$) we impose no further requirements.
        \end{enumerate}

        Finally, for $\bar{f}$ of interval length $k$ and $(\bar{x}, \bar{d^0}, \bar{d^1})$ recovered from $\bar{f}$ we define when $\bar{g} \in (\omega)^{\leq\lh(\bar{x})}$ is $(\bar{f}, \bar{x}, \bar{d^0}, \bar{d^1})$-\emph{matching}. For
        $n < k$
        define temporarily
        \begin{align*}
            \varphi(n)\ :\Longleftrightarrow  \ \,  &\text{no two } n_0, n_1 \in B_0(\bar{g} \res (n + 1), \bar{d^0} \res (n + 1), \bar{d^1} \res (n + 1))\cap n\\
            &\text{are comparable w.r.t.} <^{\bar{g}}_0, \bar{d^0} \res (n + 1), \bar{d^1} \res (n + 1) \text{ are good and}\\
                                &n \in  B(\bar{g} \res (n + 1), \bar{d^0} \res (n + 1), \bar{d^1} \res (n + 1)).
            \end{align*}
        Then $\bar{g} \in (\omega)^{\leq\lh(\bar{x})}$ is $(\bar{f}, \bar{x}, \bar{d^0}, \bar{d^1})$-\emph{matching}, if $\bar{g}$ is $\bar{x}$-compatible and for every
        $n < k$
        it holds that
        \begin{enumerate}[(i)]
                \item if $\varphi(n)$, then
                $n \in \dom(\bar{g})$ and $\bar{f}(n) = \bar{g}(n)$;
                \item if $k := \bar{g}^{-1}(n)$ is defined, $k \in \dom(\bar{f})$ and $\varphi(k)$, then $\bar{f}(n) = e(\bar{x}, \bar{d^0}, \bar{d^1})(k)$;
                \item if $k := (\bar{g}^{-1} \circ e(\bar{x}, \bar{d^0}, \bar{d^1}))(n)$ is defined, $k  \in \dom(\bar{f})$ and $\varphi(k)$, then\\
                $\bar{f}(n) = e(\bar{x}, \bar{d^0}, \bar{d^1})^2(n)$;
                \item
                if none of the ``if'' parts of the ``if \ldots then'' statements from (i), (ii) and (iii) are true, then $\bar{f}(n) = e(\bar{x}, \bar{d^0}, \bar{d^1})(n)$.
            \end{enumerate}
            Of course, the idea behind the definition of $(\bar{f}, \bar{x}, \bar{d^0}, \bar{d^1})$-matching is that it captures the requirements imposed by the definition of $\dot{e}$. This will be made more precise in Proposition \ref{sigma_range}.

        With this we are ready to establish that the set of generators of the constructed group is definable.
        
        \begin{proposition}\label{sigma_range}
            $\range(\dot{e})$ is a $\Pi^0_1$ subset of $\baire$.
        \end{proposition}
        
        \begin{proof}
            We will define a $\Delta^0_1$ set $U \sq \injfin$, so that
            \begin{equation}\label{in_range}
            \range(\Dot{e}) = \Big\{ f \in \baire \, \Big\vert \, (\forall k \in \omega)\, f \res \big(\sum_{m \leq k} | I_m|\big) \in U \Big\}.
            \end{equation}
            Then clearly $\range(\dot{e})$ will be $\Pi^0_1$.
            For $\bar{f} \in \injfin$ of interval length $k$, we let
            \begin{align*}
                \bar{f} \in U\ :\Longleftrightarrow \ \, &(\exists (\bar{x}, \bar{d^0}, \bar{d^1}) \in
                2^k \times 2^k \times 2^k)\\
                &(\bar{x}, \bar{d^0}, \bar{d^1}) \text{ is recovered from } \bar{f}\, \land\\
                &(\exists \bar{g} \in (\omega)^{\leq \lh(\bar{x})})\, \bar{g} \text{ is } (\bar{f}, \bar{x}, \bar{d^0}, \bar{d^1})\text{-matching}.
            \end{align*}
            Note that even though we are using two existential quantifiers in the above definition, $U$ is $\Delta^0_1$, as we can verify whether there are witnesses to the existential quantifiers in finitely many steps. We leave the details to the reader.
            
            We now verify that (\ref{in_range}) holds. For $k \in \omega$, let $l(k) := \sum_{m \leq k} |I_m|$. Suppose first that $f \in \range(\dot{e})$. Then there are $(x, d^0, d^1) \in 2^\omega \times 2^\omega \times 2^\omega$ for which $f = \dot{e}(x, d^0, d^1)$. Clearly, for each $k$ it holds that $(x \res k, d^0 \res k, d^1 \res k)$ is recovered from $f \res l(k)$. Let also $g := \chi^{\dagger}(x)$. We consider the following two cases:
            
            (1) If $g \in \injinf$, then clearly each $g \res k$ is $(f\res l(k), x \res k, d^0 \res k, d^1 \res k)$-matching. In particular, for every $k \in \omega$ it holds that $f \res l(k) \in U$.
            
            (2) If $g \in \injfin$, then for every $k$ with $k \leq \lh(g)$ let $g_k := g \res k$ and for every $k$ with $k > \lh(g)$ define $g_k := g$.  %
            Then each $g_k$ is clearly $(x \res k)$-compatible and $(f \res l(k), x \res k, d^0 \res k, d^1 \res k)$-matching. We have thus established that for every $k \in \omega$, $f \res l(k) \in U$.
            
            Conversely, assume that for $f \in \baire$ it holds that $(\forall k \in \omega)\, f \res l(k) \in U$. For each $k \in \omega$, let $(\bar{x}_k, \bar{d^0_k}, \bar{d^1_k})$ be recovered from $f \res l(k)$ and let $\bar{g}_k$ be $(f \res l(k), \bar{x}_k, \bar{d^0_k}, \bar{d^1_k})$-matching. Clearly, for every $k < k'$ it must hold that $\bar{x}_k \sqsubset \bar{x}_{k'}$ and $\bar{d^j_k} \sqsubset \bar{d^j_{k'}}$ for $j \in 2$. Let $x \in 2^\omega$ be unique such that $x \res k = \bar{x}_k$ and let $d^j$ be unique such that $d^j \res k = \bar{d^j_k}$. On the other hand, it is not necessarily the case that the $(\bar{g}_k)_k$ line up. Nevertheless, setting $g := \chi^\dagger(x)$, it holds that whenever $\Bar{g}_k(m)$ is used in cases (i) to (iv) in the definition of matching, we have that $m \in \dom(g)$ and $\bar{g}_k(m) = g(m)$. To see this, note that if $\Bar{g}_k(m)$ is used in cases (i) to (iv) in the definition of matching, then for every $k' > k$ it must hold that $\bar{g}_{k'}(m) = \bar{g}_k(m)$, so if $\bar{g}_k(m) \neq g(m)$, there would be some $k' > k$ for which $\bar{g}_{k'}$ is not $(f \res l(k'), \bar{x}_{k'}, \bar{d^0_{k'}}, \bar{d^1_{k'}})$-matching. Hence we may assume that $\Bar{g}_k = g \res k$ for every $k \in \omega$. The reader can routinely verify that indeed, the conditions imposed on $x, d^0, d^1$ and $g$ guarantee that $\dot{e}(x, d^0, d^1) = f$. Note that when $g \in \injfin$, the conditions (i) to (iv) from the definition of matching make sure that $f(n) = e(x, d^0, d^1)(n)$ for all $n \geq m$ for an appropriate $m$ (in the same way as in the definition of $\dot{e}$).
        \end{proof}
        
        \begin{proposition}
            The group $\G$ is $\Sigma^0_2$ and is freely generated by $\range(\dot{e})$.
        \end{proposition}
        
        \begin{proof}
            The proof of Proposition \ref{G_cofinitary} shows that $\G = \langle \range(\dot{e}) \rangle$ is freely generated.
            
            For $h \in S_\infty$, $h$ is in $\G$ precisely when there is a word $w(x_0, \ldots, x_k)$, and some $g_0, \ldots, g_k \in \range(\dot{e})$ so that
            \[
            h = w(g_0, \ldots, g_k).
            \]
            The idea is to use one existential quantifier to guess the word $w(x_0, \ldots, x_k)$ (note that there are countably many words) and then use $\delta$ and the argument from the previous proposition to recover each $g_i$ for $i \in k + 1$ in a $\Pi^0_1$ way so that $h = w(g_0, \ldots, g_k)$. Together, the complexity of membership of $\G$ is $\Sigma^0_2$. We leave the details to the reader.
        \end{proof}
        
        This was the last ingredient needed to complete the proof of Theorem \ref{constructed_mcg}.
    
    \section{Limitations to the construction of a \texorpdfstring{$G_\delta$ mcg}{G-delta}}\label{limitations}

    Since $S_\infty$ is $G_\delta$ in $\baire$, the notions of being $G_\delta$ in $S_\infty$ and in $\baire$ are the same. Moreover, by \cite[Exercise 9.6]{cdst}, for $G \leq S_\infty$ being closed in $S_\infty$ is equivalent to being $G_\delta$. From now on, we will use the expression ``$G_\delta$ mcg'' for $G \leq S_\infty$, which are $G_\delta$ (in either of the spaces) and ``closed mcg'' for $G \leq S_\infty$, which is closed as a subset of $\baire$.
    
    Since the construction of the previous section produces a freely generated mcg, the following result of Dudley (see \cite{dudley} for the original paper and Rosendal's \cite{rosendal_automatic_continuity} for an updated overview) implies that among freely generated mcgs our construction achieves the best possible complexity.
    
    \begin{theorem}[Dudley]\label{no_polish_free}
        There is no Polish topology on the free group with continuum many generators.
    \end{theorem}

    Actually, the work of Slutsky (see Theorem 1.6 of \cite{slutsky_automatic_cont}) gives an even stronger statement.

    \begin{theorem}[Slutsky]\label{no_polish_anywhere_free}
        The only Polish group topology on any free product $G * H$ is the discrete topology (in which case, of course, $G * H$  must be countable).
    \end{theorem}

    This means that if there is a $G_\delta$ mcg $G$, all its elements are ``non-free'' from the rest of the group, i.e., for any $g \in G$, there is no $H \leq G$ for which $G = \langle g \rangle * H$. Hence any construction of a $G_\delta$ mcg would have to ensure that there are non-trivial relations between its elements. Achieving this while also making sure that the group is cofinitary seems fairly difficult, but perhaps not impossible. Note that on the other hand it is possible to adapt the techniques in this paper to produce mcgs which are isomorphic to
    \[
    \Asterisk_{x \in \baire} \Z/2\Z,
    \]
    that is, the free product of continuum many copies of the two-element group, $\Z/2\Z$. Nevertheless, all currently known constructions of definable mcgs produce groups which decompose into free products.

    The authors thank Asger T\"{o}rnquist for directing our attention to Dudley's result, through which we discovered Slutsky's work.
    Previously, we developed a method of \emph{definable gluing of orbits}. The idea was the following. It is not hard to prove that there can be no $G_\delta$ cofinitary group $G$ which has exactly one $k$-orbit for all $k \geq 1$ (here we count as $k$-orbits of $G$ the orbits under the diagonal action of $G$ on $(\omega)^k$). We developed a method of gluing $k$-orbits (for all $k \geq 1$) while preserving a certain notion of definability and while also preserving cofinitariness. Then we proved that for an mcg, which has a nice catching function (this is a natural condition based on the known constructions of mcgs, and which implies that the group is $G_\delta$) and which has a small degree of freeness, we can define from it a $G_\delta$ mcg with just one $k$-orbit for all $k \geq 1$, resulting in a contradiction. We do not describe this idea in further detail, as Theorem \ref{no_polish_anywhere_free} implies a stronger statement.

    \section{Open problems}\label{open_problems}
    
    The most important open question is the longstanding:
    
    \begin{question}\label{q_G_delta}
        Is there a $G_\delta$ mcg?
    \end{question}

    If one is convinced that the answer to Question \ref{q_G_delta} is negative, then a great step forward is to first show that the answer to the next open question is negative.
    
    \begin{question}
        Is there a closed (in the Baire space) mcg?
    \end{question}
    
    Note that in \cite{kastermans_orbits}, Kastermans proved that every mcg has only finitely many orbits. For $k \geq 1$ and $G \leq S_\infty$, the \emph{$k$-orbits} of $G$ are the orbits of the diagonal action of $G$ on $(\omega)^k$. The usual orbits are hence 1-orbits. It is unknown whether the following is possible (this was already asked in \cite{kastermans_orbits}).
    
    \begin{question}
        Can there be an mcg with infinitely many $k$-orbits for some $k > 1$?
    \end{question}
    
    Call $x \in (\omega \cup \{\omega\})^\omega$ \emph{increasing} if for every $n \in \omega$ it holds that $x(n + 1) \geq x(n)$. For an mcg $G$, define $x_G \in \baire$ by setting $x_G(0) := 0$ and for $n \geq 1$ that
    \[
    x_G(n) := \text{number of } n \text{-orbits of } G.
    \]
    
    \begin{question}
        Which increasing $x \in (\omega \cup \{\omega\})^\omega$ arise as $x_G$ for some mcg $G$?
    \end{question}

    \subsection{Maximal finitely periodic groups}
    
    Say that $g \in S_\infty \setminus \{\id_\omega\}$ is \emph{finitely periodic} if $\langle g \rangle \sq S_\infty$ has finitely many finite orbits. A subgroup $G \leq S_\infty$ is then called \emph{finitely periodic}, if every $g \in G \setminus \{\id_\omega\}$ is finitely periodic. Clearly, every finite periodic group is also cofinitary. We use the abbreviation \emph{mpg} to refer to maximal finitely periodic groups. Of course, the first question asked by any descriptive-set-theorist is the following.

    \begin{question}
        Is there a Borel mpg?
    \end{question}

    Note that our construction of a $\Sigma^0_2$ mcg cannot be immediately adapted to this setting, as the construction relies on the sequence of finite intervals $(I_n)_{n \in \omega}$. In particular, for any $x \in 2^\omega \setminus \range(\chi)$ and any $d^0, d^1 \in 2^\omega$, we have that $\dot{e}(x, d^0, d^1)$ is not finitely periodic. Nevertheless, the authors believe that the answer to the question is positive.

    We next present a slight variant of a notion appearing in Kastermans' \cite{kastermans_compact} and \cite{kastermans_orbits}. For $G \leq S_\infty$ and $A \sq G$ with $\id_\omega \in A$ (for convenience), we write $W_A(x)$ for the set of all finite words in the variable $x$ and elements of $A$. Formally,
    \[
        W_A(x) = \big\{ g_k x^{i_k} g_{k-1} x^{i_{k - 1}} \cdots x^{i_0} g_0 \, \big| \, g_0, \ldots, g_k \in A, i_0, \dots, i_k \in \Z \setminus \{0\} \big\}.
    \]
    For $h \in S_\infty$ and $w(x) \in W_A(x)$, we let $w(h)$ be the permutation obtained by substituting all occurrences of $x$ by $h$.

    \begin{theorem}\label{no_infinite_orbits}
        There is no mpg with infinitely many orbits.
    \end{theorem}

    \begin{proof}
        The proof is based on the proof of Kastermans' Theorem 13 from \cite{kastermans_orbits}. Due to the similarity, we omit some details.
        
        Let $(O_i)_{i \in \omega}$ be some fixed enumeration of all orbits of a finitely periodic group $G$. We define $h \in S_\infty$ inductively through increasingly larger finite approximations. Let $h_0 := \emptyset$ and suppose we have defined a finite partial injective function $h_k$ for some $k \in \omega$. Then let
        \[
        n := \min\big((\omega \setminus \dom(h_k)) \cup (\omega \setminus \range(h_k))\big),
        \]
        $j$ be the least such that $O_j \cap (\dom(h_k) \cup \range(h_k) \cup \{n\}) = \emptyset$ and $m := \min O_j$. If $n \notin \dom(h_k)$, set $h_{k+1} := h_k \cup \{(n, m)\}$, otherwise set $h_{k+1} := h_k \cup \{(m, n)\}$. By definition,
        \[
        h := \cup \{h_k \, | \, k \in \omega\} \in S_\infty
        \]
        and since the orbit structure of $\langle G \cup \{h\}\rangle$ is different from the orbit structure of $G$, it follows that $h \notin G$.

        Take any reduced $ w(x) \in W_G \setminus \{\emptyset\}$; we will show that $w(h)$ is finitely periodic.
        Note that when we reduce $w(x)^2$, strictly less than half of each instance of $w(x)$ gets annihilated (as otherwise $w(x) = \emptyset$). We consider the following cases:
        \begin{enumerate}
            \item $w(x)$ is of odd length $2 l + 1$, where $g \in G$ is at the $(l+1)$-th place and when reducing $w(x)^2$, we annihilate $l$ characters from each instance of the word, so that now there is $g^2$ at the $(l+1)$-th place in the reduced $w(x)^2$;
            \item same as case (1), but with $x^i$ (where $i \in \Z \setminus \{0\}$) in place of $g$;
            \item when we reduce $w(x)^2$, the resulting word is strictly longer than $w(x)$.
        \end{enumerate}
        In order so that we do not have to separately consider all three different cases, we observe that the following captures all of them: there is a finite set $A \sq G$ and $g \in G$ so that for every $k \in \omega$, the reduced form of $w(x)^k$ is equal to some word $w_k(x) \in W_{A \cup \{g^k\}}$.

        Suppose that $n \in \omega$ is in a finite orbit of $\langle w(h) \rangle$. Then there is some $k_n$, so that $n$ is a fixed point of $w_{k_n}(h)$. Using the argument of Theorem 13 from \cite{kastermans_orbits}, there is some $g_n \in A \cup
        \{g\}$, so that if $g_n \in A$ then it has a fixed point $m_n$, and if $g_n = g$ then $m_n$ is a fixed point of $g^{k_n}$. Similarly as in \cite{kastermans_orbits}, we observe that for $n_0, n_1$ which belong to different finite orbits of $\langle w(h)\rangle$, either $g_{n_0} \neq g_{n_1}$, or $g_{n_0} = g_{n_1}$ and $m_{n_0} \neq m_{n_1}$ (to be precise, like in \cite{kastermans_orbits}, we need to consider each $g_{n_i}$ together with its position in $w_{k_{n_i}}(x)$). By the Pigeonhole principle, if there are infinitely many finite orbits of $\langle w(h)\rangle$, there must either be some $f \in A$ which has infinitely many fixed points or $g$ has infinitely many finite orbits. This contradicts the assumption on $G$. Since $w(x)$ was arbitrary, we have proved that $G \cup \{h\}$ generates a finitely periodic group. In particular, $G$ is not maximal.
    \end{proof}

    Recall, that a set $A \sq \baire$ is \emph{eventually bounded} if there is some $x \in \baire$, so that for every $y \in A$ there is some $n \in \omega$ such that for every $m > n$ it holds that $y(m) < x(m)$. Kastermans proved in \cite{kastermans_compact} (see Theorem 10) that no mcg can be eventually bounded.

    \begin{question}
        Can an mpg be eventually bounded?
    \end{question}

    At first sight it seems likely that an adaptation of the proof of Theorem 10 from \cite{kastermans_compact}, similar to the way we adapted Kastermans' argument in the proof of Theorem \ref{no_infinite_orbits}, would be successful in proving that no mpg can be eventually bounded. However, it turns out that the straightforward modification does not work this time. The authors nevertheless believe the answer should be negative.

    If $G$ is an mcg, which is finitely periodic, it is clearly an mpg. The authors believe however that it is possible to find an example, affirmatively answering the following question.
    
    \begin{question}
        Is there an mpg, which is not an mcg?
    \end{question}

    \bibliographystyle{alpha}
    \bibliography{bibliography}

\end{document}